\documentclass[english]{msjproc}
\usepackage{lmodern}

\usepackage{etex}
\usepackage{amsmath,amsfonts,amssymb,amsthm,amscd,mathrsfs,mathtools}
\usepackage{pifont,bm,enumitem}
\usepackage{xcolor}
\usepackage{ytableau}
\ytableausetup{boxsize=1.15em}
\usepackage[all]{xy}

\setlist{nolistsep}
\AtBeginDocument{
      \setlength\abovedisplayskip{2pt}
      \setlength\belowdisplayskip{2pt}}
\makeatletter

\newcommand{\FF}{\mathbb{F}}
\newcommand{\CC}{\mathbb{C}}

\newcommand{\NN}{\mathbb{N}}
\newcommand{\kk}{\Bbbk}%
\newcommand{\Map}{\operatorname{Map}}

\newcommand{\Id}{\operatorname{Id}}

\renewcommand{\ge}{\geqslant}
\renewcommand{\le}{\leqslant}
\newcommand{\osigma}{\boldsymbol{\sigma}}
\newcommand{\osigmas}{\osigma^M}
\newcommand{\osigman}{\osigma^N}
\newcommand{\ow}{\overline{w}}

\newcommand{\bw}{\boldsymbol{w}}
\newcommand{\bww}{\boldsymbol{w'}}

\newcommand{\emptyc}{\boldsymbol{e}_C}
\newcommand{\emptyr}{\boldsymbol{e}_R}

\newcommand{\T}{\mathcal{T}}
\newcommand{\C}{\mathcal{C}}
\newcommand{\R}{\mathcal{R}}
\newcommand{\Ins}{\mathcal{I}}

\newcommand{\TotSh}{\mathcal{S}}

\newcommand{\YT}{\mathbf{YT}}
\newcommand{\Cols}{\mathbf{Col}}
\newcommand{\Rows}{\mathbf{Row}}

\newcommand{\Coll}{\mathbf{Col}^{\bullet}}
\newcommand{\Roww}{\mathbf{Row}^{\bullet}}
\newcommand{\CCol}{\mathbf{C}}
\newcommand{\Pl}{\mathbf{Pl}}
\newcommand{\Norm}{\mathbf{Norm}}

\newcommand{\Mon}{\mathbf{M}}

\newcommand{\oMon}{{\Mon}}
\newcommand{\oNorm}{{\Norm}}
\newcommand{\cols}{\operatorname{cols}}
\newcommand{\rows}{\operatorname{rows}}

\newcommand\longmapsfrom{\mathrel{\reflectbox{\ensuremath{\longmapsto}}}}
\newcommand\lessrot{\mathbin{\rotatebox[origin=c]{90}{\ensuremath{<}}}}
\newcommand\dotover{\leavevmode\cleaders\hb@xt@ .22em{\hss $\cdot$\hss}\hfill\kern\z@}
\newcommand{\dotfrac}[2]{
\ooalign{$\genfrac{}{}{0pt}{0}{#1}{#2}$\cr\dotover\cr}
}

\definecolor{mygreen}{rgb}{0,.455,0} 
\definecolor{myviolet}{rgb}{.45,.05,.545}
\definecolor{myred}{rgb}{.545,0,0}
\definecolor{myblue}{rgb}{.024,.15,.645} 
\colorlet{MyRed}{myred!40!white}
\colorlet{MyBlue}{myblue!40!white}
\colorlet{MyGreen}{mygreen!40!white}
\colorlet{MyViolet}{myviolet!40!white}
\definecolor{MyGrey}{rgb}{.804,.804,.756} 
\usepackage[colorlinks=false,linkbordercolor=MyGrey,citebordercolor=MyGrey,
urlbordercolor=MyGrey, pdfauthor={V. Lebed}, pdftitle={Plactic monoids}]{hyperref}

  \theoremstyle{plain}
\newtheorem{theorem}{Theorem}
\newtheorem{conjecture}{Conjecture}
 
\newtheorem{proposition}{Proposition}[section]
\newtheorem{corollary}[proposition]{Corollary}
\newtheorem{lemma}[proposition]{Lemma}
  \theoremstyle{remark}
\newtheorem{observation}[proposition]{Observation}
\newtheorem{remark}[proposition]{Remark}
  \theoremstyle{definition}
\newtheorem{definition}[proposition]{Definition}
\newtheorem{notation}[proposition]{Notation}
\newtheorem{example}[proposition]{Example}

\begin{document}
  \title{Plactic monoids: a braided approach}
  \author{Victoria LEBED}{Trinity College Dublin}
  \email{lebed.victoria@gmail.com}
  \webpage{\nolinkurl{http://www.maths.tcd.ie/~lebed/index.html}}

  \maketitle

  \begin{abstract}
Young tableaux carry an associative product, described by the Schensted algorithm. They thus form a monoid $\mathbf{Pl}$, called \emph{plactic}. It is central in numerous combinatorial and algebraic applications. In this paper, the tableaux product is shown to be completely determined by a braiding $\sigma$ on the (much simpler!) set of columns $\mathbf{Col}$. Here a \emph{braiding} is a set-theoretic solution to the Yang--Baxter equation. As an application, we identify the Hochschild cohomology of $\mathbf{Pl}$, which resists classical approaches, with the more accessible braided cohomology of $(\mathbf{Col},\sigma)$. The cohomological dimension of $\mathbf{Pl}$ is obtained as a corollary. Also, the braiding~$\sigma$ is proved to commute with the classical crystal reflection operators~$s_i$.
  \end{abstract}

  \section{Introduction}

\ytableausetup{centertableaux} 

A \emph{Young tableau} on a totally ordered alphabet~$A$ is a finite decreasing sequence of non-empty $A$-rows. Here \emph{$A$-rows} are non-decreasing words $\bw \in A^*$, partially ordered by the relation
\begin{align}\label{E:RowOrder}
x_1 \ldots x_s \underset{R}{\succ} y_1 \ldots y_t \qquad &\Longleftrightarrow \qquad s \le t \quad \text{and} \quad \forall 1 \le i \le s, \, x_i > y_i.
\end{align}
Introduced by Young in 1900, these combinatorial gadgets play a crucial role in the representation theory of the symmetric groups~$S_n$ and the complex general linear groups $GL_n(\CC)$, in the Schubert calculus of Grassmannians, and in the study of symmetric functions. Figure~\ref{P:YT} shows the graphical representation of a Young tableau.\footnote{We use the French rather than the English notation, where smaller rows are placed at the top. It makes column reading more natural---from top to bottom.}. 
 
\begin{figure}[!h]\centering
\vspace*{-10mm}
\[T = \ytableaushort{3,266,134} \hspace*{3cm} \parbox[c][7em][c]{0.3\textwidth}{$\R(T) = 3\;266\;134$ \\ \; \\ $\C(T) = 321\;63\;64$}\]
\vspace*{-10mm}
\caption{A Young tableau with its row-wise and column-wise readings}\label{P:YT}
\end{figure}

The tableaux are read either row-wise or column-wise, yielding two injective maps
$\R,\C \colon \YT_A \hookrightarrow A^*$ 
from the set of Young tableaux on~$A$ to the word monoid on~$A$; the definition of these maps is clear from the example in Figure~\ref{P:YT}, where $A = \NN$. Alternatively, Young tableaux can be defined as finite non-decreasing sequences of non-empty \emph{$A$-columns} (= decreasing words), with respect to the partial order
\begin{align}\label{E:ColOrder}
x_1 \ldots x_s \underset{C}{\preccurlyeq} y_1 \ldots y_t \qquad \Longleftrightarrow \qquad s \ge t \quad \text{and} \quad \forall 1 \le i \le t, \, x_{i+s-t} \le y_i.
\end{align}
Motivated by the problem of longest non-decreasing subwords for $\bw \in A^*$, Schensted~\cite{Schensted} proposed an algorithm\footnote{An essentially equivalent version of this algorithm appeared in an earlier but long ignored work of Robinson~\cite{Robinson}.} for inserting new entries into a Young tableau, realized by an {insertion map} $\Ins \colon \YT_A \times A^* \to \YT_A$. 
His algorithm is recalled in Section~\ref{S:Algo}. It comes with several bonuses. First, it endows~$\YT_A$ with the associative product $T \ast T' := \Ins(T, \R(T'))$. Second, it is used to construct a Young tableau out of any word $\bw \in A^*$, via the {tableau map} $\T \colon  \bw \mapsto \Ins(\emptyset,\bw)$. This map is surjective, with sections $\R$ and $\C$. The tableau $\T(\bw)$ contains useful information about the word~$\bw$. In particular, the length of its last row yields the maximal length of a non-decreasing subword of~$\bw$, answering the original question of Schensted. More generally, Greene~\cite{placticGreene} showed the lengths of $k$ last rows (or $k$ first columns) of $\T(\bw)$ to encode the maximal length of a subword of~$\bw$ forming a shuffle of $k$ non-decreasing (respectively, decreasing) words. 

Knuth~\cite{placticKnuth} deduced from $\T$ a bijection between the monoid $(\YT_A,\ast)$ and the quotient $\Pl_A$ of~$A^*$ by the relations
\begin{align}\label{E:Knuth}
xzy &\sim zxy, \qquad x \le y < z;& yxz &\sim yzx, \qquad x < y \le z.
\end{align}
This quotient was baptized \emph{plactic monoid}\footnote{With some imagination, relations~\eqref{E:Knuth} evoke plate tectonics, \emph{tectonique des plaques} in French.} and thoroughly explored by Lascoux and Sch{\"u}tzenberger~\cite{placticLS}, followed by numerous other researchers. Among its multiple applications are one of the first proofs of the Littlewood--Richardson rule\footnote{famous for its long announce-a-proof-and-wait-for-an-error-to-be-found history} \cite{LRrule}, which is used to compute products of Schur functions, intersections of Grassmannians, and tensor products of irreducible representations of $GL_n(\CC)$ or $S_n$; and a combinatorial description of the Kostka--Foulkes polynomials~\cite{FoulkesLS}, which appear in the representation theory of $GL_n(\FF_q)$ as well as in certain lattice models in statistical mechanics. Plactic monoids are also strongly related with crystal bases, which describe the behavior of quantum groups when their deformation parameter $q$ tends to~$0$ \cite{placticQuantum,placticQuantumKT}. See~\cite{plactic} for a beautiful self-contained overview of different facets of plactic monoids. 

The very recent work of Cain--Gray--Malheiro \cite{placticGSbasis} and Bokut--Chen--Chen--Li \cite{placticGSbasis2} launched a revival of plactic monoids. They independently showed that the set $\Coll_A$ of non-empty columns in~$A^*$ (alternatively, the set $\Roww_A$ of non-empty rows) forms a {Gr\"obner--Shirshov basis} for~$\Pl_A$. In~\cite{Lopatkin}, Lopatkin fed the column basis into the algebraic discrete Morse machinery in order to compute 
the Hochschild cohomology of the algebras $\kk \Pl_A$ (here $\kk$ is a field). At the heart of his work lies a study of the restriction of the product~$\ast$ to $\Cols_A^{\times 2}$. Here the column set $\Cols_A = \Coll_A \sqcup \{\emptyc\}$ is seen inside~$\YT_A$, and the empty column~$\emptyc$ is identified with the empty tableau. Lopatkin observed that the $\ast$-product of two one-column tableaux has at most two columns. This yields an operator\footnote{In fact Lopatkin worked on $(\Coll_A)^{\times 2}$, which made his~$\sigma_C$ only partially defined and added a lot of special cases to his definitions and proofs.} $\sigma_C$ on $\Cols_A^{\times 2}$, which turns out to be a \emph{braiding}, i.e., a (non-invertible) solution to the \emph{Yang--Baxter equation}
\begin{equation}\label{E:YBE}
(\sigma \times \Id)  (\Id \times \sigma)  (\sigma \times \Id) = (\Id \times \sigma)  (\sigma \times \Id)  (\Id \times \sigma)
\end{equation}
on $\Cols_A^{\times 3}$. This equation plays a fundamental role in mathematical areas ranging from statistical mechanics to quantum field theory, from low-dimensional topology to quantum group theory. Attention to its set-theoretic form dates back to Drinfel$'$d~\cite{DrST}.

In Section~\ref{S:PlacticBraiding} we extend the braiding~$\sigma_C$ to the much larger set $\YT_A^e:=\YT_A \times \NN_0$ of $\NN_0$-decorated Young tableaux---or, equivalently, to~$\Pl_A^e:=\Pl_A \times \NN_0$. We also propose its row version $\sigma_R$, defined on $\Rows_A = \Roww_A \sqcup \{\emptyr\}$ and on the whole~$\YT_A^e$. The $\NN_0$-decorations keep track of empty columns or rows. Such ``dummy'' elements are recurrent in normalization problems (cf. \cite{HessOz,DehGui,LebedIdempot}); the $\ast$-product, and hence our braidings, can be seen as particular instances of normalization. We recover the undecorated plactic monoid~$\Pl_A$ as the {structure monoid} for~$\sigma_C$ or~$\sigma_R$ (Theorem~\ref{T:PlIsStrMonoid}). The  \emph{structure monoid} of a braiding is a quadratic monoid classically associated to it, in a way that captures many of its properties.

Both~$\sigma_C$ and~$\sigma_R$ are compatible with the (concatenation$\times$sum) product on~$\Pl_A^e$, and yield two braided commutative monoid structures (Theorem~\ref{T:BraidingPlactic}). Moreover, they define two actions of the \emph{positive braid monoid}~$B_k^+$ \eqref{E:Bn} on $(\Pl_A^e)^{\times k}$. Even better: $\sigma_C$ is idempotent on $\Cols_A^{\times 2}$, and so is $\sigma_R$ on $\Rows_A^{\times 2}$. Basics on idempotent braidings are recalled in Sections~\ref{S:IdempotentBraiding}; see~\cite{LebedIdempot} for more details and examples. As is always the case for idempotent braidings, the $B_k^+$-actions given by $\sigma_C$ and~$\sigma_R$ restrict to the actions of the \emph{$0$-Hecke monoid} $C_k$~\eqref{E:Cn} on $\Cols_A^{\times k}$ and $\Rows_A^{\times k}$ respectively. Also known as \emph{Coxeter monoids}\footnote{These monoids were defined and studied for all Coxeter groups. Since only the symmetric group case is relevant for us, we use simplified terms and notations.}, the $C_k$ appeared in the work of Tsaranov~\cite{Tsaranov}, and since then were applied to Hecke algebras, to the Bruhat order on Coxeter groups, to Tits buildings, and to planar graphs \cite{FomGr,HST0Hecke,DolanTrimble,GanMaz,Kenney,Kenney2}.

Lascoux and Sch{\"u}tzenberger~\cite{LS88} introduced the operators $s_i$  which in every Young tableau on $A_n = \{1,2,\ldots,n\}$ replace some carefully chosen $i$s with $i+1$, and vice versa.  The~$s_i$ are sometimes referred to as \emph{coplactic}, or \emph{crystal reflection} operators. They yield an $S_n$-action on~$\Pl_{A_n}$, instrumental in some applications of Young tableaux~\cite{plactic}. In Section~\ref{S:Permutation} (Theorem~\ref{T:ActionCompat}) we show that this $S_n$-action, extended to $(\Pl_{A_n}^e)^{\times k}$, commutes with the $B_k^+$-action given by~$\sigma_C$ or~$\sigma_R$.

In~\cite{Lebed1}, the author developed a cohomology theory for braidings---and more generally, for solutions to the YBE in any preadditive monoidal category. See \cite{FRS_BirackHom,HomologyYB,Eisermann,BirackHom} for alternative approaches, and~\cite{Lebed1,LebedVendramin,PrzWang} for their comparison. This \emph{braided cohomology theory} unifies the cohomological study of basic algebraic structures: associative and Lie algebras, self-distributive structures, bialgebras, Hopf (bi)modules, Yetter-Drinfel$'$d modules, factorized monoids, etc. It also suggests nicely behaved theories for new algebraic structures, such as cycle sets~\cite{LebedVendramin}. Moreover, it comes with a handy graphical calculus, replacing technical verifications. Section~\ref{S:BrHom} is a reminder on the braided cohomology theory for idempotent set-theoretic braidings. In this particular case, the braided cohomology coincides with the Hochschild cohomology of the structure monoid of the braiding. Moreover, they are isomorphic as graded algebras when the coefficients allow cup products to be defined (see \cite{LebedIdempot}, or Theorem~\ref{T:BrHomIdempot}). Differential complexes are much smaller and simpler on the braided side, yielding an efficient tool for computing Hochschild cohomology. 

As an application, in Section~\ref{S:BrHomPlactic} Hochschild cohomology computations for plactic monoids are substituted with the simpler braided cohomology computations for the column braiding~$\sigma_C$. This allows us to identify a copy of the exterior algebra~$ \Lambda (\kk A)$ inside the Hochschild cohomology $H^* (\Pl_A; \kk)$ of~$\Pl_A$ with trivial coefficients when $A$ is finite (Theorem~\ref{T:CohomPlactic}). Conjecturally, this exterior algebra covers the whole cohomology. We thus simplify and sharpen computations from~\cite{Lopatkin}, including them into the conceptual framework of braided cohomology. Computations with different coefficients (Theorem~\ref{T:CohomPlactic2}) allow us to determine the cohomological dimension of~$\Pl_A$, in the Hochschild sense. It is $1$, $3$, or $\infty$, for $A$ of size $1$, $2$, or $> 2$ respectively.

\medskip
\textbf{Acknowledgments.} The author is grateful to Patrick Dehornoy for bringing her attention to the mysterious appearance of braidings in Viktor Lopatkin's work on plactic monoids; to Viktor Lopatkin for patient guidance through that work; and to Vladimir Dotsenko for the fruitful suggestion to vary coefficients. This work was partially supported by the program ANR-11-LABX-0020-01 (Henri Lebesgue Center, University of Nantes), and by a Hamilton Research Fellowship (Hamilton Mathematics Institute, Trinity College Dublin). 

\section{Schensted algorithm}\label{S:Algo}

Fix an ordered alphabet~$A$. Take a Young tableau~$T$ over~$A$ and an $x \in A$. We now recall two versions of the Schensted algorithm for inserting $x$ into~$T$, with  basic applications.

The right insertion algorithm works as follows. If~$T$ is empty or $x$ is at least as large as the last element of the last row of~$T$, then attach~$x$ to the right of this last row. Otherwise choose in the last row the leftmost element~$x'$ among those exceeding~$x$; replace it with~$x$; and insert~$x'$ into the tableau obtained from~$T$ by forgetting its last row, using the same procedure. An easy verification shows that the algorithm indeed produces a Young tableau. An example is treated in Figure~\ref{P:Schensted}.

\begin{figure}[!h]
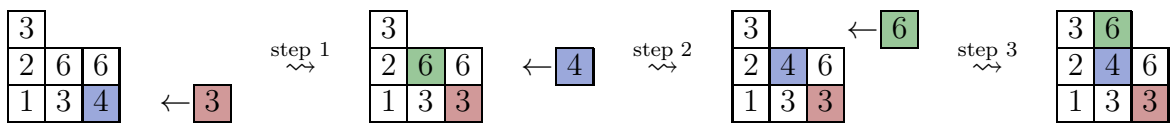
\centering
\vspace*{-.3cm}
\[\ytableaushort{3,266,13{*(MyBlue)4}\none{\none[\leftarrow]}{*(MyRed)3}} \quad \overset{\text{step }1}{\leadsto} \quad \ytableaushort{3,2{*(MyGreen)6}6\none{\none[\leftarrow]}{*(MyBlue)4},13{*(MyRed)3}} \quad \overset{\text{step }2}{\leadsto} \quad
\ytableaushort{3\none\none{\none[\leftarrow]}{*(MyGreen)6},2{*(MyBlue)4}6,13{*(MyRed)3}} \quad \overset{\text{step }3}{\leadsto} \quad \ytableaushort{3{*(MyGreen)6},2{*(MyBlue)4}6,13{*(MyRed)3}}\]
\vspace*{-.3cm}
\caption{Inserting the element~$3$ into a Young tableau from the right in three steps}\label{P:Schensted}
\end{figure}

The Schensted algorithm thus yields a map $\YT_A \times A \to \YT_A$ and, by iteration, the \emph{insertion map}
\begin{align*}
\Ins \colon \YT_A \times A^* &\to \YT_A.
\end{align*}
Specializing the first argument to~$\emptyset$, one gets the \emph{tableau map} 
\begin{align*}
\T \colon A^* &\to \YT_A.
\end{align*}
Also recall the \emph{row} and \emph{column maps} 
\begin{align*}
\R,\C \colon \YT_A &\to A^*
\end{align*}
from Figure~\ref{P:YT}. Finally, consider the \emph{Knuth equivalence}~$\sim$ generated by the three-letter relations~\eqref{E:Knuth}, and the quotient \emph{plactic monoid} $\Pl_A := A^* / \mathord\sim$.

The following lemma summarizes the basic properties of these maps and relations. Its proof is purely combinatorial and elementary; it can be found, for instance, in~\cite{plactic}.

\begin{lemma}\label{L:Insertion map}
\begin{enumerate}
\item The map~$\Ins$ descends to the quotient  $\YT_A \times \Pl_A \to \YT_A$.
\item The maps $\R$ and $\C$ are sections of $\T$: one has $\T  \R = \T  \C = \Id_{\YT_A}$.
\item The composition $\R  \T$ is Knuth-equivalent to the identity: for all $\bw \in Pl_A$, one has $\R  \T (\bw) \sim \bw$.
\end{enumerate}
\end{lemma}

These properties lead to the following fundamental result:

\begin{theorem}[Knuth, \cite{placticKnuth}]\label{T:Knuth}
The tableau map~$\T$ induces a bijection between~$\Pl_A$ and~$\YT_A$, with the inverse induced by the row map~$\R$ (equivalently, by the column map~$\C$). The concatenation product on~$\Pl_A$ corresponds under this bijection to the associative product
\begin{align}\label{E:Ast}
T \ast T' &= \Ins(T, \R(T'))
\end{align}
on~$\YT_A$, for which the empty tableau is the unit element.
\end{theorem}

Now, the product~$\ast$ from the theorem suggests one more way of inserting an element~$x$ into a Young tableau~$T$. Namely, one can view~$x$ as a one-cell tableau and compute $x \ast T$. By iteration, this gives the \emph{left insertion map}
\begin{align*}
\Ins' \colon A^* \times \YT_A &\to \YT_A.
\end{align*}
The associativity of~$\ast$ and the properties from Lemma~\ref{L:Insertion map} imply that $\Ins'$ descends to the quotient $\Pl_A \times \YT_A$, and that the corresponding left tableau map $\bw \mapsto \Ins'(\bw,\emptyset)$ coincides with~$\T$. 

We next give a combinatorial description of the operation $x \ast T$, which can be called the \emph{left Schensted algorithm}. If~$T$ is empty or $x$ is larger than the first element of the first column of~$T$, then attach~$x$ to the top of this first column. Otherwise choose in the first column the smallest element $x' \ge x$, replace it with~$x$, and then insert~$x'$ into the tableau obtained from~$T$ by forgetting its first column, using the same procedure. An example is treated in Figure~\ref{P:SchenstedLeft}.

\begin{figure}[!h]
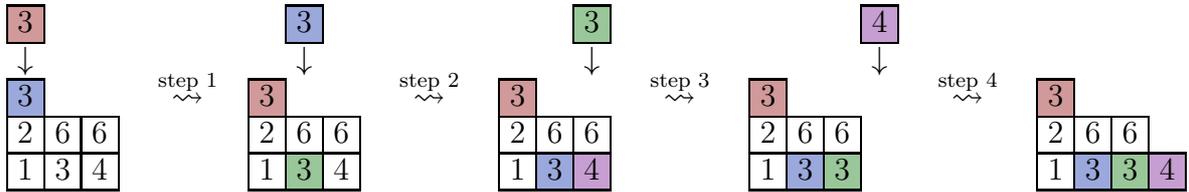
\centering 
\hspace*{-.6cm}
\begin{align*}
\ytableaushort{{*(MyRed)3},{\none[\downarrow]},{*(MyBlue)3},266,134} \quad \overset{\text{step }1}{\leadsto} \quad &\ytableaushort{\none{*(MyBlue)3},\none{\none[\downarrow]},{*(MyRed)3},266,1{*(MyGreen)3}4} \quad \overset{\text{step }2}{\leadsto} \quad \ytableaushort{\none\none{*(MyGreen)3},\none\none{\none[\downarrow]},{*(MyRed)3},266,1{*(MyBlue)3}{*(MyViolet)4}} \quad \overset{\text{step }3}{\leadsto} \quad \ytableaushort{\none\none\none{*(MyViolet)4},\none\none\none{\none[\downarrow]},{*(MyRed)3},266,1{*(MyBlue)3}{*(MyGreen)3}} \quad \overset{\text{step }4}{\leadsto} \quad \ytableaushort{\none,\none,{*(MyRed)3},266,1{*(MyBlue)3}{*(MyGreen)3}{*(MyViolet)4}}
\end{align*}
\hspace*{-.6cm}
\caption{Inserting the element~$3$ into a Young tableau from the left in four steps}\label{P:SchenstedLeft}
\end{figure}

\begin{remark}\label{R:ColRowDuality}
The two versions of the Schensted algorithm are in fact polar cases of the same general procedure. Consider binary relations $\triangleleft_1,\triangleleft_2$ on a set~$A$ which 
\begin{itemize}
   \item are multi-transitive: $x \triangleleft_i y \triangleleft_j z$ implies $x \triangleleft_i z$ and $x \triangleleft_j z$ for all $i,j \in \{1,2\}$;
   \item satisfy the law of the excluded third: for all $x,y \in A$, exactly one of $x \triangleleft_1 y$ and $y \triangleleft_2 x$ holds. 
\end{itemize}
Define Young tableaux on such data $(A,\triangleleft_1,\triangleleft_2)$, called \emph{admissible data}, as in the ordered case, with~$\triangleleft_1$ replacing~$<$ and~$\triangleleft_2$ replacing~$\le$. Schensted's recipes, repeated verbatim, describe how to insert elements into such Young tableaux on the right and on the left. Now, the data $(A,\triangleleft_2,\triangleleft_1)$ are also admissible, and yield tableaux which are the transposes of those for $(A,\triangleleft_1,\triangleleft_2)$. The right Schensted algorithm for $(A,\triangleleft_2,\triangleleft_1)$ corresponds to the left one for $(A,\triangleleft_1,\triangleleft_2)$. To recover the case of an ordered~$A$, take $<$ as $\triangleleft_1$ and $\le$ as~$\triangleleft_2$. More generally, for an ordered set $A = A_1 \sqcup A_2$ split into two, the relations 
\[x \triangleleft_i y \qquad \Longleftrightarrow \qquad (x < y \;\text{ or }\; x=y \in A_i)\]
are admissible. On the other hand, starting from any admissible $(A,\triangleleft_1,\triangleleft_2)$ and identifying all $x,y \in A$ for which $x \triangleleft_i y \triangleleft_i x$ holds for some~$i$, one gets induced relations on the quotient which are precisely such ``split-order'' relations.
\end{remark}

We finish with one more elementary property of the insertion algorithms:

\begin{lemma}\label{L:CountRowCol}
The $\ast$-product of two one-row tableaux contains no more than two rows. The same property holds for columns.
\end{lemma}

\begin{proof}
Use the Schensted algorithm to subsequently insert the entries $x_1,x_2,\ldots$ of a row ${r'}$ into a one-row tableau ${r}=y_1 y_2 \ldots$. Each~$x_i$ is placed into the cell occupied by the leftmost letter $y_{k(i)}$ in~${r}$ which is greater than~$x_i$ (excluding the already replaced letters $y_{k(j)}$, $j<i$) if there is one, or to the right of this row if not. Indeed, $x_{i+1}$ is at least as big as~$x_i$, and hence will replace someone to its right. This argument also gives $k(i) < k(i+1)$, hence $y_{k(i)} \le y_{k(i+1)}$. Thus all the chased letters $y_{k(i)}$ are attached to the second row from below. So, the tableau ${r} \ast {r'} = \Ins({r}, {r'})$ contains at most two rows. An example is treated in Figure~\ref{P:Prod2Col}. The column case is analogous. \qedhere

\begin{figure}[!h]
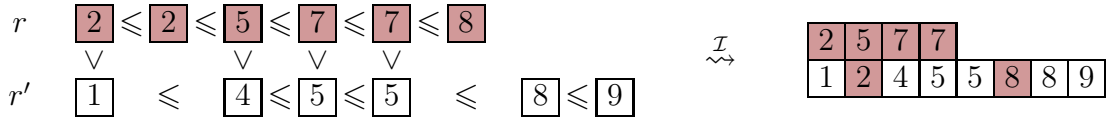
\centering 
\begin{align*}
\ytableaushort{{\none[{r}]}\none{*(MyRed)2}{\none[\le]}{*(MyRed)2}{\none[\le]}{*(MyRed)5}{\none[\le]}{*(MyRed)7}{\none[\le]}{*(MyRed)7}{\none[\le]}{*(MyRed)8},
\none\none{\none[\lessrot]}\none\none\none{\none[\lessrot]}\none{\none[\lessrot]}\none{\none[\lessrot]},
{\none[{r'}]}\none 1\none{\none[\le]}\none 4 {\none[\le]} 5 {\none[\le]} 5 \none{\none[\le]}\none 8 {\none[\le]} 9}
&\qquad\overset{\Ins}{\leadsto}\qquad \ytableaushort{{*(MyRed)2}{*(MyRed)5}{*(MyRed)7}{*(MyRed)7}, 1{*(MyRed)2}455{*(MyRed)8}89}
\end{align*}
\caption{The $\ast$-product of two rows}\label{P:Prod2Col}
\end{figure}
\end{proof}

\section{Idempotent braidings}\label{S:IdempotentBraiding}

In the next section, we will interpret the two versions of the Schensted algorithm in terms of idempotent braidings on the set of $A$-rows (respectively, $A$-columns). But before that we need to recall basic properties of this type of braidings; see~\cite{LebedIdempot} for a detailed exposition with proofs and multiple examples.

A \emph{braided set} is a set~$X$ endowed with a \emph{braiding}, i.e., a map $\sigma \colon X^{\times 2} \to X^{\times 2}$ satisfying the \emph{YBE}~\eqref{E:YBE}. An \emph{idempotent} braiding obeys the additional axiom $\sigma  \sigma = \sigma$. A braiding~$\sigma$ classically extends from~$X$ to words in~$X$; the resulting braiding on~$X^*$ is denoted by~$\osigma$. The set of \emph{normal} words for a given braided set is defined as
\[ \Norm(X,\sigma) = \{ \, x_{1}\ldots x_{k} \in X^* \;\, | \;\, \forall 1 \le j < k, \; \sigma (x_{j},x_{{j+1}}) = (x_{j},x_{{j+1}}) \,\}.\]

In real-life examples, braidings often interact with other structure on the underlying set. For instance, a \emph{braided commutative monoid} is a monoid $(M,\cdot,1)$ endowed with a braiding~$\sigma$, subject to the following compatibility conditions for all $u,v,w \in M$:
\begin{align}
\sigma (u \cdot v,w) &= (w'',u' \cdot v'), &&\text{where } \sigma(v,w)=(w',v'), \; \sigma(u,w')=(w'',u');\label{E:BrMonoid}\\
\sigma (u,v \cdot w) &= (v' \cdot w',u''), &&\text{where } \sigma(u,v)=(v',u'), \; \sigma(u',w)=(w',u'');\label{E:BrMonoid'}\\
w' \cdot v' &= v \cdot w, &&\text{where } \sigma(v,w)=(w',v');\label{E:BrMonoid''}\\
\sigma(v,w) &= (w,v), &&\text{where } v=1 \text{ or } w=1.\label{E:BrMonoid'''}
\end{align}
One recovers usual commutative monoids taking as~$\sigma$ the flip $(u,v) \mapsto (v,u)$. We will turn (a decorated version of) plactic monoids into braided commutative monoids in two ways---using an extension of the row or the column braiding.

A braiding induces an action of the \emph{positive braid monoid} 
\begin{equation}\label{E:Bn}
B_k^+ = \langle \, b_1, \ldots, b_{k-1} \, | \, b_ib_j=b_jb_i \text{ for } |i-j| >1,  \, b_ib_{i+1}b_i = b_{i+1}b_ib_{i+1} \,\rangle^+
\end{equation}
on~$X^{\times k}$, for all $k \in \NN$, via
\begin{equation}\label{E:BnActs}
b_i \mapsto \Id_X^{\times (i-1)} \times \sigma \times \Id_X^{\times (k-i-1)}.
\end{equation}
In the idempotent case, this action descends to the quotient
\begin{equation}\label{E:Cn}
C_k = \langle \, b_1, \ldots, b_{k-1} \, | \, b_ib_j=b_jb_i \text{ for } |i-j| >1,  \, b_ib_{i+1}b_i = b_{i+1}b_ib_{i+1}, \, b_ib_i=b_i \,\rangle^+
\end{equation}
of~$B_k^+$, referred to as the \emph{Coxeter monoid}.

An idempotent braided set $(X,\sigma)$ is called \emph{pseudo-unital}, or \emph{PUIBS}, if endowed with a \emph{pseudo-unit}, i.e., an element $1 \in X$ satisfying:
\begin{enumerate}
\item both $\sigma(1,x)$ and $\sigma(x,1)$ lie in $\{\, (1,x),(x,1) \,\}$ for all $x\in X$;
\item dropping any occurrence of~$1$ from a normal word, one still gets a normal word.
\end{enumerate}
Given a word $\bw \in X^*$, let the word $\boldsymbol{\ow}$ be obtained from it by erasing all its letters $1$. This yields a projection $\Norm(X,\sigma) \twoheadrightarrow \oNorm(X,\sigma,1)$, where $\oNorm(X,\sigma,1)$ is the set of normal words without the letter~$1$, called \emph{reduced normal words}. 
In our example of $A$-rows and $A$-columns, the empty row or column will be a pseudo-unit, and reduced normal words will correspond to Young tableaux.

The \emph{structure monoid} of a braided set $(X,\sigma)$ is presented as follows:
\[\Mon(X,\sigma) = \langle \, X \; | \; xy = y'x' \text{ whenever } \sigma(x,y)=(y',x'), \, x,y \in X \,\rangle^+.\] 
To get from it the \emph{reduced structure monoid} $\oMon(X,\sigma,1)$ of a PUIBS $(X,\sigma,1)$, identify the letter~$1$ with the empty word. A representative $x_{1}\ldots x_{k}$ of an element of $\Mon(X,\sigma)$, with $x_{j} \in X$, is called its \emph{normal form} if it is a normal word. \emph{Reduced normal form} is defined similarly. Structure monoids should be thought of as ``universal enveloping monoids'' of a braiding; in particular, they encode the representation theory of $(X,\sigma,1)$, in the sense of \cite{Lebed1,LebedIdempot}. This construction brings group-theoretic tools into the study of the YBE, and is the basis of most current approaches to the classification of braidings. In the opposite direction, it yields a rich source of (semi)groups and algebras with interesting algebraic properties (see \cite{GIBergh,ESS,Rump,JesOknI,Chouraqui,DehYBE} and references therein). Our aim here is to recover plactic monoids as the reduced structure monoids of row/column braidings, and to apply to them general results on structure monoids, especially on their cohomology.

The languages of (reduced) normal words and (reduced) structure monoids turn out to be equivalent:

\begin{theorem}[\cite{LebedIdempot}]\label{T:NormalVsStrMon}
\begin{enumerate}
\item The structure monoid $\Mon(X,\sigma)$ of a braided set $(X,\sigma)$ is braided commutative, with the braiding~$\osigmas$ induced by the braiding~$\osigma$ on~$X^*$.
\item If~$\sigma$ is idempotent, then the tautological map $\Norm(X,\sigma) \to \Mon(X,\sigma)$ is bijective. Its inverse sends an $m \in \Mon(X,\sigma)$ represented by a word $\bw \in X^{k}$ to ${\Delta_k \bw}$, which turns out to be the unique normal form of~$m$:
\begin{align*}
\Norm(X,\sigma) &\longleftrightarrow \Mon(X,\sigma),\\
\operatorname{Taut} \,:\, \bw &\longmapsto [\bw],\\
{\Delta_k \bw} &\longmapsfrom [\bw] \,:\, \operatorname{NForm}.
\end{align*}
Here the longest element $\Delta_k = b_1 (b_2b_1) \cdots (b_{k-1} \cdots b_2b_1)$ of~$C_k$ acts on~$X^{k}$ via~\eqref{E:BnActs}.
\item Under the above bijection, the braiding~$\osigmas$ on $\Mon(X,\sigma)$ corresponds to the restriction~$\osigman$ of~$\osigma$ to $\Norm(X,\sigma)$. Further, the concatenation product on $\Mon(X,\sigma)$ corresponds to a product~$\ast$ on $\Norm(X,\sigma)$, which can be computed as the braiding~$\osigman$ followed by concatenation.
\item For a PUIBS $(X,\sigma,1)$, the above bijection induces a bijection $\oNorm(X,\sigma,1) \leftrightarrow \oMon(X,\sigma,1)$. Its inverse computes the (unique) reduced normal form.
\end{enumerate}
\end{theorem}

As a result of the last assertion, the concatenation product on $\oMon(X,\sigma,1)$ pulls back to an associative product on $\oNorm(X,\sigma,1)$, still denoted by~$\ast$. Explicitly, for reduced normal words $\boldsymbol{v},\bw$ of length~$n$ and~$m$, one has $\boldsymbol{v} \ast \bw = \overline{\Delta_{n+m}(\boldsymbol{vw})}$. The empty word is a unit for~$\ast$. In general the braiding~$\osigman$ does not restrict to $\Norm(X,\sigma,1)$.

\section{Two braidings on the plactic monoid}\label{S:PlacticBraiding}

We now return to Young tableaux and the plactic monoid on an ordered alphabet~$A$. 

Recall the row set $\Rows_A = \Roww_A \sqcup \{\emptyr\}$ including the empty row~$\emptyr$.  We will often switch between its interpretations inside the word monoid $A^*$, and inside $\YT_A$ (as the subset of one- or zero-row tableaux). Our aim is to define a braiding~$\sigma_R$ on~$\Rows_A$. Take $r_1, r_2 \in \Rows_A \hookrightarrow \YT_A$. According to Lemma~\ref{L:CountRowCol}, there are two possibilities for the tableau $r_1 \ast r_2$:
\begin{enumerate}
\item it can have two non-empty rows $r'_2,r'_1$, in which case we put $\sigma_R(r_1, r_2)=(r'_2,r'_1)$;
\item it can be a single row---the concatenation $r_1r_2$, which can be empty; we then declare $\sigma_R (r_1, r_2)=(\emptyr,r_1r_2)$.
\end{enumerate}
One obtains an idempotent operator on $\Rows_A^{\times 2}$. There is an analogous operator~$\sigma_C$ on $\Cols_A^{\times 2}$. Both somewhat ignore the empty row/column:
\begin{align}\label{E:Empty}
& \sigma_R(r,\emptyr)=\sigma_R(\emptyr,r)=(\emptyr,r),&
& \sigma_C(c,\emptyc)=\sigma_C(\emptyc,c)=(c,\emptyc).
\end{align}

\begin{proposition}\label{PR:RowBr}
The operators~$\sigma_R$ and~$\sigma_C$ above are idempotent braidings on the row set $\Rows_A$ and the column set $\Cols_A$ respectively.
\end{proposition}

\begin{proof}
One should check the YBE~\eqref{E:YBE} on $\Rows_A^{\times 3}$ and $\Cols_A^{\times 3}$. Take three rows $r_1,r_2,r_3$. We will prove that, evaluated on $(r_1,r_2,r_3)$, both sides of the YBE for~$\sigma_R$ yield the rows of $r_1 \ast r_2 \ast r_3$ (read from top to bottom), preceded with some empty rows~$\emptyr$ if necessary.

\textbf{Left-hand side:} Let us show that the three (possibly empty) rows of $(\sigma_R \times \Id) (\Id \times \sigma_R)(r'_2,r'_1,r_3)$ form a Young tableau, knowing that the rows $(r'_2,r'_1) = \sigma_R (r_1,r_2)$ form one. Write the rows as words in~$A$: $r'_2 = z_1 \ldots z_m$, $r'_1 = y_1 \ldots y_l$, $r_3 = x_1 \ldots x_k$ (Figure~\ref{P:YBERows}). Put $\sigma_R (r'_1,r_3) = (r'_3,r''_1)$, $\sigma_R (r'_2,r'_3) = (r''_3,r''_2)$. Extend the row ordering $\underset{R}{\succ}$ from~\eqref{E:RowOrder} to empty rows by declaring $\emptyr \underset{R}{\succ} r$ for any row~$r$. The property $r''_3 \underset{R}{\succ} r''_2$ being automatic, we are left with $r''_2$ and~$r''_1$. As explained in the proof of Lemma~\ref{L:CountRowCol}, $r''_1$ is obtained from~$r'_1$ by replacing some letters $y_{i_1}, \ldots, y_{i_p}$ with $x_1, \ldots, x_p$ (where $x_j < y_{i_j}$, $x_j \ge y_{i_j-1}$) and adding the remaining letters of~$r_3$ to the end. The chased letters $y_{i_1} \ldots y_{i_p}$ assemble into the row~$r'_3$. Since $m \le l$, $x_j < y_{i_j} < z_{i_j}$, and $y_i < z_i$ for the non-chased letters~$y_i$, one has $r'_2 \underset{R}{\succ} r''_1$. Now, the row~$r''_2$ is obtained from~$r'_2$ by replacing some of its letters~$z_{h_s}$ with~$y_{i_s}$, and adding the remaining letters~$y_{i_t}$ of~$r'_3$ to the end, to positions~$h_t$. The relations $y_{i} < z_{i}$ yield $h_j \le i_j$. But then $y'_{h_j} \le y'_{i_j} = x_j < y_{i_j}$, where $y'_1y'_2\ldots$ is the row~$r''_1$. For the non-chased letters~$z_i$, one has $y'_i < z_i$ because of $r'_2 \underset{R}{\succ} r''_1$. Put together, this yields $r''_2 \underset{R}{\succ} r''_1$. Hence the rows $r''_3,r''_2, r''_1$ possibly include some empty ones, followed by the rows of a Young tableau~$T$. In particular, $\T(r''_3 r''_2 r''_1)=T$. Since at the level of~$A^*$ the braiding~$\sigma_R$ is the composition~$\R \T$, Lemma~\ref{L:Insertion map} implies $r_1r_2r_3 \sim r''_3 r''_2 r''_1$. Theorem~\ref{T:Knuth} then yields $r_1 \ast r_2 \ast r_3 = \T(r_1r_2r_3) = \T(r''_3 r''_2 r''_1)=T$, as desired.

\ytableausetup{boxsize=1.4em}
\begin{figure}[!h]
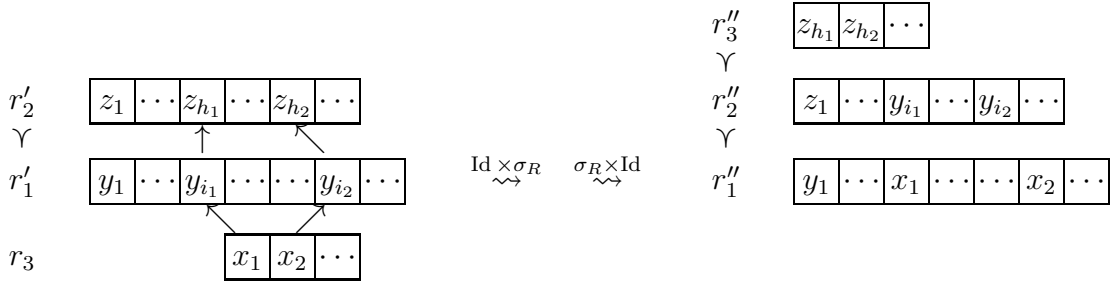
\centering 
\begin{align*}
& && \ytableaushort{{\none[r''_3]}\none{z_{h_1}}{z_{h_2}}{\cdots}}\\[-7pt]
& && \ytableaushort{{\none[\curlyvee]}}\\[-7pt]
&\ytableaushort{{\none[r'_2]}\none{z_1}{\cdots}{z_{h_1}}{\cdots}{z_{h_2}}{\cdots}} && \ytableaushort{{\none[r''_2]}\none{z_1}{\cdots}{y_{i_1}}{\cdots}{y_{i_2}}{\cdots}}\\[-7pt]
&\ytableaushort{{\none[\curlyvee]}\none\none\none{\none[\uparrow]}\none}\hspace*{7pt}
\ytableaushort{{\none[\nwarrow]}} && \ytableaushort{{\none[\curlyvee]}}\\[-7pt]
&\ytableaushort{{\none[r'_1]}\none{y_1}{\cdots}{y_{i_1}}{\cdots}{\cdots}{y_{i_2}}{\cdots}} & \quad \overset{\Id \times \sigma_R}{\leadsto} \quad \overset{\sigma_R \times \Id}{\leadsto} \qquad& 
\ytableaushort{{\none[r''_1]}\none{y_1}{\cdots}{x_1}{\cdots}{\cdots}{x_2}{\cdots}}\\[-7pt]
&\hspace*{7pt}\ytableaushort{\none\none\none\none{\none[\nwarrow]}}\hspace*{17pt}
\ytableaushort{{\none[\nearrow]}} &&\\[-7pt]
&\ytableaushort{{\none[r_3]}\none\none\none\none{x_1}{x_2}{\cdots}} &&
\end{align*}
\caption{The left-hand side of the YBE for three rows}\label{P:YBERows}
\end{figure}
\ytableausetup{boxsize=1.15em}

\textbf{Right-hand side:} Put $(r'_3,r'_2) = \sigma_R (r_2,r_3)$, $(r''_3,r'_1) = \sigma_R (r_1,r'_3)$, $(r''_2,r''_1) = \sigma_R (r'_1,r'_2)$. One has to check the relation $r''_3 \underset{R}{\succ} r''_2$; the reasoning can then be completed as for the left-hand side. The word~$r''_3$ is formed by the letters $z_{i_1},\ldots, z_{i_k}$ of~$r_1$ chased by the letters $y_1,\ldots,y_k$ of~$r'_3$. Since $r'_3 \underset{R}{\succ} r'_2$, the first letters $x_1,\ldots,x_k$ of~$r'_2$ satisfy $x_j < y_j$, and thus in the computation of $\sigma_R (r'_1,r'_2)$ they chase the letters $y'_{h_1},\ldots, y'_{h_k}$ of $r'_1 = y'_1y'_2\ldots$ with $h_j \le i_j$, possibly plus some extra letters on the right. The inequalities $y'_{h_j} \le y'_{i_j} =y_j < z_{i_j}$ then yield $r''_3 \underset{R}{\succ} r''_2$.

The column case can be treated analogously. Alternatively, one can use the duality argument from Remark~\ref{R:ColRowDuality}.
\end{proof}

The explicit description of the $\ast$-product from the proof of Lemma~\ref{L:CountRowCol} yields a useful comparison of rows/columns before and after the braiding procedure:

\begin{observation}\label{O:Comparison}
\begin{enumerate}
\item Take $r_1,r_2 \in \Rows_A$, and put $\sigma_R(r_1,r_2) = (r_3,r_4)$. Then $r_3$ is a subrow of~$r_1$, $r_2$ is a subrow of~$r_4$, and one has $r_3 \underset{R}{\succ} r_4$, $r_3 \underset{R}{\succ} r_2$. Moreover, $r_i \underset{R}{\succcurlyeq} r_j$ holds when $i=3$ or $j=4$, where the relation~$\underset{R}{\succcurlyeq}$ is defined in the obvious way.
\item Take $c_1,c_2 \in \Cols_A$, and put $\sigma_C(c_1,c_2) = (c_3,c_4)$. Then $c_1$ is a subcolumn of~$c_3$, $c_4$ is a subcolumn of~$c_2$, and one has $c_i \underset{C}{\preccurlyeq} c_j$ whenever $i=3$ or $j=4$.
\end{enumerate}
\end{observation}

The column braiding has one more useful elementary property, generalizing~\eqref{E:Empty}:

\begin{observation}\label{O:SubCol}
Let $c_1$ be a subcolumn of~$c_2$. Then one has
\[\sigma_C(c_1,c_2) = \sigma_C(c_2,c_1) = (c_2,c_1).\]
Even better: Observation~\ref{O:Comparison} implies that $\sigma_C(c_1,c_2)$ is of the form $(c_2,c'_1)$ or $(c'_2,c_1)$ if and only if $c_1$ is a subcolumn of~$c_2$.
\end{observation}

There is no hope for a similar characterization for subrows. In fact, even the evaluation of~$\sigma_R$ on the diagonal $(r_1,r_1)$ is rather involved.

These observations suggest seeing $\oMon(\Cols_A,\sigma_C,\emptyc)$ and $\oMon(\Rows_A,\sigma_R,\emptyr)$ as some sort of \emph{``idempotent'' monoids of $I$-type}; cf. \cite{GISkew,GIBergh} for a more classical ``involutive'' version of this notion.

\begin{example}
For the two-element alphabet $A = \{1,2\}$, with $1 < 2$, there are only four columns: $\emptyc, \ytableaushort{1}, \ytableaushort{2}, \ytableaushort{2,1}$. For~$\sigma_C$, the only values not given by Observation~\ref{O:SubCol} are
\begin{align*}
\sigma_C(\ytableaushort{1},\ytableaushort{2}) &= (\ytableaushort{1},\ytableaushort{2}), & \sigma_C(\ytableaushort{2},\ytableaushort{1}) = (\ytableaushort{2,1},\emptyc).
\end{align*}
The rows are in bijection with $\NN_0^{\times 2}$: the couple $(k,l)$ corresponds to $k$ ones followed by $l$ twos. The braiding~$\sigma_R$ then reads
\begin{align*}
\sigma_R ((k_1,l_1),(k_2,l_2)) &= ((0, \min\{l_1,k_2\}),(k_1 + k_2, l_1 + l_2 -\min\{l_1,k_2\})).
\end{align*}
\end{example}

The row and column braidings are far from being invertible. However, a weak form of invertibility does hold for them:

\begin{observation}\label{O:WeakInvert}
Put $\sigma (a_1,a_2) = (a_3,a_4)$, where $\sigma$ is either $\sigma_R$ or $\sigma_C$, and the $a_i$ are either rows or columns. Then any three of the $a_i$s determine the remaining one. In the row case, $a_1$ and $a_4$ suffice to recover $a_2$ and $a_3$. In the column case, the only couple that determines all the $a_i$s is $(a_1,a_2)$.
\end{observation}

Recalling the definition of normal words, one easily describes them for $\sigma = \sigma_R$ or~$\sigma_C$, using the properties~\eqref{E:Empty} of empty rows and columns:
\begin{align}
\Norm(\Rows_A,\sigma_R) & \; \overset{1:1}{\longrightarrow} \; \YT_A \times \NN_{0},\label{E:NormRow}\\ 
(\emptyr^\alpha r_1 \ldots r_k) & \; \longmapsto \; (\, \dotfrac{r_1}{r_k}\, , \alpha \, );\notag\\[3pt] 
\Norm(\Cols_A,\sigma_C) & \; \overset{1:1}{\longrightarrow} \; \YT_A \times \NN_{0},\label{E:NormCol}\\ 
(c_1 \ldots c_k \emptyc^\alpha) & \; \longmapsto \; (\, c_1 \dotfrac{\quad}{\quad} c_k \, , \alpha \, ), \notag
\end{align}
where on the left one has $\emptyr \neq r_1 \underset{R}{\succ} \ldots \underset{R}{\succ} r_k$ and $c_1 \underset{C}{\preccurlyeq} \ldots \underset{C}{\preccurlyeq} c_k \neq \emptyc$ respectively. Working with columns is preferable for certain applications, since for finite~$A$ the column set $\Cols_A$ is of finite size~$2^{|A|}$, while $\Rows_A$ is always infinite.

Theorem~\ref{T:NormalVsStrMon} now provides braided commutative monoid structures on $\YT_A \times \NN_{0}$. 

\begin{notation}
The number of rows and columns in a Young tableau~$T$ is denoted by $\rows(T)$ and $\cols(T)$ respectively.
\end{notation}

\begin{theorem}\label{T:BraidingPlactic}
Take a totally ordered alphabet~$A$. The set $\YT_A \times \NN_{0}$ of $\NN_{0}$-decorated Young tableaux on~$A$ can be seen as a braided commutative monoid in two ways:
\begin{enumerate}
\item The unit is the element $(\emptyset,0)$, the product is given by 
\begin{align*}
(T_1,\alpha_1) \ast_R (T_2,\alpha_2) &= (T_1 \ast T_2,\alpha_1+\alpha_2 + \rows(T_1)+\rows(T_2) - \rows(T_1 \ast T_2)),
\end{align*}
and the braiding $\osigma_R ((T_1,\alpha_1), (T_2,\alpha_2)) = ((T'_2,\alpha'_2), (T'_1,\alpha'_1))$ is uniquely defined by the following conditions:
\begin{itemize}
\item the rows of~$T'_2$ followed by those of~$T'_1$ form the tableau $T_1 \ast T_2$;
\item $\alpha'_1+ \rows(T'_1) = \alpha_1+ \rows(T_1)$, $\alpha'_2+ \rows(T'_2) = \alpha_2+ \rows(T_2)$;
\item $\alpha'_1 > 0$ if and only if $\rows(T_1 \ast T_2) < \alpha_1+ \rows(T_1)$, in which case $T'_2 = \emptyset$.
\end{itemize}
\item The unit is the element $(\emptyset,0)$, the product is given by 
\begin{align*}
(T_1,\alpha_1) \ast_C (T_2,\alpha_2) &= (T_1 \ast T_2,\alpha_1+\alpha_2 + \cols(T_1)+\cols(T_2) - \cols(T_1 \ast T_2)),
\end{align*}
and the braiding $\osigma_C ((T_1,\alpha_1), (T_2,\alpha_2)) = ((T'_2,\alpha'_2), (T'_1,\alpha'_1))$ is uniquely defined by the following conditions:
\begin{itemize}
\item the columns of~$T'_2$ followed by those of~$T'_1$ form the tableau $T_1 \ast T_2$;
\item $\alpha'_1+ \cols(T'_1) = \alpha_1+ \cols(T_1)$, $\alpha'_2+ \cols(T'_2) = \alpha_2+ \cols(T_2)$;
\item $\alpha'_2 > 0$ if and only if $\cols(T_1 \ast T_2) < \alpha_2+ \cols(T_2)$, in which case $T'_1 = \emptyset$.
\end{itemize}
\end{enumerate}
\end{theorem}

We will call the braidings from the theorem \emph{insertion braidings}, since they are based on Schensted's insertion algorithms.

For the arguments with $\alpha_1=\alpha_2=0$, the formulas for $\ast_R$ and~$\ast_C$ immediately give

\begin{corollary}
The functions $\rows$ and $\cols$ are subadditive on~$\YT_A$: for two tableaux $T_1,T_2$ one has
\begin{align*}
\rows(T_1 \ast T_2) &\le \rows(T_1)+\rows(T_2), &\cols(T_1 \ast T_2) &\le \cols(T_1)+\cols(T_2).
\end{align*}
\end{corollary}

Knuth's bijection (Theorem~\ref{T:Knuth}) allows one to transport the above structures to the $\NN_{0}$-decorated plactic monoid:
\begin{corollary}
The constructions from Theorem~\ref{T:BraidingPlactic} induce, via the tableau map~$\T$, two braided commutative monoid structures on $\Pl_A \times \NN_{0}$.
\end{corollary}

\begin{remark}\label{R:InsertionAsBraiding}
According to Theorem~\ref{T:NormalVsStrMon}, the braiding~$\osigma_R$ suffices to reconstruct the product~$\ast_R$ on $\YT_A \times \NN_{0}$, and hence the product~$\ast$ on~$\YT_A$. Schensted algorithms being nothing else than recipes for calculating $T \ast \ytableaushort{a}$ and $\ytableaushort{a} \ast T$ for a tableau~$T$ and an $a \in A$, we thus include these algorithms into the braided paradigm.
\end{remark}

At this stage working with the $\NN_{0}$-decorations might seem a handicap. We will now propose two ways to circumvent them. The second one will allow us to study the cohomology of undecorated plactic monoids in Section~\ref{S:BrHomPlactic}.

\begin{remark}\label{R:RestrictedTableaux}
The braiding $\osigma_R$ restricts to the set of couples $(T,\alpha)$ with $\rows(T)+\alpha = m$ for a fixed $m \in \NN$. Such couples correspond to $T \in \YT_A$ with $\rows(T) \le m$. One obtains an idempotent braiding on the set $\YT_A^{r \le m}$ of such tableaux. Similarly, $\osigma_C$ restricts to an idempotent braiding on the set $\YT_A^{c \le m}$ of tableaux~$T$ with $\cols(T) \le m$. The case $m=1$ recovers the braidings~$\sigma_R$ and~$\sigma_C$. In contrast, note that the total braidings $\osigma_R$ and $\osigma_C$ are not idempotent, but satisfy a weaker condition $\sigma^{3} = \sigma$.
\end{remark}

Alternatively, the explicit description \eqref{E:NormRow}-\eqref{E:NormCol} of normal words and the properties~\eqref{E:Empty} of $\emptyr$ and~$\emptyc$ show that they are pseudo-units for $(\Rows_A,\sigma_R)$ and $(\Cols_A,\sigma_C)$ respectively. Theorem~\ref{T:NormalVsStrMon} then implies
\begin{theorem}\label{T:PlIsStrMonoid}
The plactic monoid on a totally ordered alphabet~$A$ is isomorphic to the reduced structure monoids for the row/column braidings for~$A$: 
\begin{align*}
\oMon(\Rows_A,\sigma_R,\emptyr) & \; \overset{1:1}{\longleftrightarrow} \; \Pl_A \; \overset{1:1}{\longleftrightarrow} \; \oMon(\Cols_A,\sigma_C,\emptyc).
\end{align*}
\end{theorem}

Note that the braidings from Theorem~\ref{T:BraidingPlactic} do not descend to~$\Pl_A$.

Our braided interpretation yields a simple way to determine the center of~$\Pl_A$:

\begin{proposition}
The center of the plactic monoid~$\Pl_A$ is the free monoid generated by the longest column $c_A:=a_n a_{n-1} \ldots a_2 a_1$ when~$A$ is the finite alphabet $a_1 < a_2 < \ldots < a_n$, and is trivial when~$A$ is infinite.
\end{proposition}

\begin{proof}
Our proof jungles two interpretations of columns: inside~$\Pl_A$, and inside~$\YT_A$. By construction, the column braiding preserves the column product: $\sigma_C(c_1,c_2) = (c_3,c_4)$ implies the relation $c_3 c_4 = \C\T(c_1 c_2) = c_1 c_2$ in~$\Pl_A$. Then, by Observation~\ref{O:SubCol}, a column commutes with its subcolumns. Hence the longest column~$c_A$, if exists, is central in~$\Pl_A$. A length argument shows that $c_A$ generates a free submonoid of~$\Pl_A$. 

Now take a central element $\bw \in \Pl_A$, written in its column form $c_1 \cdots c_k$, $c_i \underset{C}{\preccurlyeq}  c_{i+1}$, using the map $\C \T$. Let us show that the last column~$c_k$, and hence every column~$c_i$, contains every letter $a \in A$. Suppose the contrary, and consider the commutation relation $\bw a=a \bw$. Theorems~\ref{T:NormalVsStrMon} and~\ref{T:PlIsStrMonoid} allow to rewrite it as a Young tableaux equality:
\[\overline{b_1 \cdots b_k (c_1, \ldots, c_k, a)} =  (c_1, \ldots, c_k) \ast a = a \ast (c_1, \ldots, c_k) = \overline{b_k \cdots b_1 (a,c_1, \ldots, c_k)},\]
where the positive braid monoid~$B^+_{k+1}$ from~\eqref{E:Bn} acts on $\Cols_A^{\times (k+1)}$ via~$\sigma_C$, and the overlines stand for the reduction, i.e., empty column elimination. Observations~\ref{O:Comparison} and~\ref{O:SubCol} leave us with two options. Either the above tableau~$T$ has $k+1$ columns, in which case the last one should be~$a$ and, at the same time, a letter from~$c_k$; but $a \notin c_k$. Or $T$ has $k$ columns, and the last one is simultaneously a subcolumn of $c_k a$ (which must be a column), and $c_k$ with a column with $\le 1$ letters placed on top of it. This is possible only when this last column is~$c_k$. Repeating this argument for all columns of~$T$, we conclude that its first column is $c_1a$ and, at the same time, $ac_1$---contradiction.
\end{proof}

This result is classical~\cite{placticLS}. In~\cite{COPlactic} it was used to show that the center of the plactic algebra $\kk \Pl_A$ is $\kk [c_A]$ when $A$ is finite.

\section{Insertion braidings vs. letter permuting operators}\label{S:Permutation}

In this section, $A$ is the alphabet $A_n := \{1,2,\ldots,n\}$ with the usual order, with $n \in \NN \cup \{\infty\}$. We will recall the classical operators $s_i$ on~$\YT_{A_n}$ which define an $S_n$-action on it, and establish its compatibility with the insertion braidings.

The~$s_i$ are first defined on words $w \in A_n^*$. Write out all the letters $i$ and~$i+1$ occurring in~$w$, keeping their order. Then match any consecutive letters $(i+1,i)$, in this order, and forget them; here the couple $(\text{last letter of } w, \text{first letter of } w)$ is considered consecutive. Continue this process for unmatched letters until only some $i$s or some $i+1$s are left. The pairs matched in the end are easily seen to be independent of the matching order. Then change all unmatched letters to $i+1$ or to $i$ respectively. Insert the obtained letters (both matched and unmatched) into their original positions, to get the word $s_i(w)$. The whole process is illustrated in~Figure~\ref{P:si}. 
 
\begin{figure}[!h]\centering
\vspace*{-.6cm}
\begin{align*}
31&2321232223311 \;\leadsto\; 1221222211 \;\leadsto\; 12\smash[b]{\color{gray}\underbracket[0.5pt][2pt]{21}}222211  \;\leadsto\; 12\smash[b]{\color{gray}\underbracket[0.5pt][2pt]{21}}222\smash[b]{\color{gray}\underbracket[0.5pt][2pt]{21}}1\\
&\leadsto\; 12\smash[b]{\color{gray}\underbracket[0.5pt][2pt]{21}}22\smash[b]{\color{gray}\underbracket[0.5pt][4pt]{2\smash[b]{\underbracket[0.5pt][2pt]{21}}1}} \;\leadsto\; \smash[b]{\color{gray}\underbracket[0.5pt][6pt]{1}}2\smash[b]{\color{gray}\underbracket[0.5pt][2pt]{21}}2\smash[b]{\color{gray}\underbracket[0.5pt][6pt]{2\smash[b]{\color{gray}\underbracket[0.5pt][4pt]{2\smash[b]{\underbracket[0.5pt][2pt]{21}}1}}}} \;\leadsto\; \smash[b]{\color{gray}\underbracket[0.5pt][6pt]{1}}{\color{red}1}\smash[b]{\color{gray}\underbracket[0.5pt][2pt]{21}}{\color{red}1}\smash[b]{\color{gray}\underbracket[0.5pt][6pt]{2\smash[b]{\underbracket[0.5pt][4pt]{2\smash[b]{\underbracket[0.5pt][2pt]{21}}1}}}} \;\leadsto\; 31{\color{red}1}321{\color{red}1}32223311
\end{align*}
\vspace*{-.6cm}
\caption{Computing $s_1(312321232223311)$}\label{P:si}
\end{figure}

The operators~$s_i$ are clearly involutive, and commute for $|i_1-i_2| > 1$. With some more work, they are shown to satisfy the YBE-like relation $s_is_{i+1}s_i = s_{i+1}s_is_{i+1}$, see \cite{placticLS,plactic}. They behave nicely with respect to the Knuth equivalence~$\sim$:

\begin{lemma}\label{L:SiKnuth}
Take four words $\boldsymbol{w_1},\boldsymbol{w_2},\boldsymbol{w_3},\boldsymbol{w_4} \in A_n^*$, with $\boldsymbol{w_2} \sim \boldsymbol{w_3}$. Then there exist four words $\boldsymbol{\widetilde{w_j}}$ of the same length as the~$\boldsymbol{w_j}$, with $\boldsymbol{\widetilde{w_2}} \sim \boldsymbol{\widetilde{w_3}}$, such that
\begin{align*}
s_i(\boldsymbol{w_1w_2w_4}) &= \boldsymbol{\widetilde{w_1}\widetilde{w_2}\widetilde{w_4}}, &s_i(\boldsymbol{w_1w_3w_4}) &= \boldsymbol{\widetilde{w_1}\widetilde{w_3}\widetilde{w_4}}.
\end{align*}
\end{lemma}

\begin{proof}
Write $s_i(\boldsymbol{w_1w_2w_4}) = \boldsymbol{\widetilde{w_1}\widetilde{w_2}\widetilde{w_4}}$, where the whole word is cut into three so that the subwords~$\boldsymbol{w_j}$ and~$\boldsymbol{\widetilde{w_j}}$ have the same length. Further, suppose that~$\boldsymbol{w_3}$ is obtained from~$\boldsymbol{w_2}$ by applying a single Knuth relation; the general case follows by induction. We regard this relation as a permutation~$\theta$ of the letters of~$\boldsymbol{w_2}$. That is, $\boldsymbol{w_3}= \theta(\boldsymbol{w_2})$.  In the notations of~\eqref{E:Knuth}, we take as $\theta$ the transposition of~$x$ and~$z$, except for two cases:
\begin{itemize}
\item the relation $i (i+1) i \sim (i+1) i i$ is seen as moving~$i$ across $(i+1) i$;
\item the relation $(i+1) i (i+1) \sim (i+1) (i+1) i$ is seen as moving~$i+1$ across $(i+1) i$.
\end{itemize}
With this interpretation, all the letters $i,i+1$ from $\boldsymbol{w_1w_2w_4}$ conserve their matched/un\-match\-ed status in $\boldsymbol{w_1}\theta(\boldsymbol{w_2})\boldsymbol{w_4}$, with respect to the matching procedure from the definition of~$s_i$. Moreover, if all the unmatched $i$s (or $i+1$s) are changed to $i+1$ (or $i$) to get $s_i(\boldsymbol{w_1w_2w_4})$, the permutation~$\theta$ still describes a valid Knuth relation for~$\boldsymbol{\widetilde{w_2}}$, and satisfies $s_i(\boldsymbol{w_1\theta(w_2)w_4}) = \boldsymbol{\widetilde{w_1}}\theta(\boldsymbol{\widetilde{w_2}})\boldsymbol{\widetilde{w_4}}$. The assertion of the lemma then holds for $\boldsymbol{\widetilde{w_3}} = \theta(\boldsymbol{\widetilde{w_2}}) \sim \boldsymbol{\widetilde{w_2}}$.
\end{proof}

Plugging empty words~$\boldsymbol{w_1}$ and~$\boldsymbol{w_4}$ into the lemma, one recovers

\begin{theorem}[\cite{placticLS,plactic}]\label{T:SnAction}
The operators~$s_i$ above, for $1 \le i < n$, define an action of the symmetric group~$S_n$ on the set of words $A_n^*$, which descends to the plactic monoid~$\Pl_{A_n}$.
\end{theorem}

The induced operators from the theorem, as well as the operators on~$\YT_{A_n}$ corresponding to them via the Knuth bijection, are still denoted by~$s_i$.

\begin{example}
The row set $\Rows_{A_n}$ is identified with~$\NN_{0}^n$: the $i$th component of $\NN_{0}^n$ counts the number of $i$s in our row. The inclusion of $\NN_{0}^n \simeq \Rows_{A_n}$ into~$A_n^*$ or into~$\YT_{A_n}$ yields sections for the corresponding abelianization maps. Restricted to $\Rows_{A_n}$, the~$s_i$ above permute the components $i$ and $i+1$ of~$\NN_{0}^n$. 
\end{example}

This example justifies the name \emph{letter permuting operators} we employ for the~$s_i$.

\begin{example}
The set~$\YT_{A_2}$ is parametrized by triples $(n_{1,2},n_{2,1},n_{2,2}) \in \NN_{0}^{3}$ with $n_{1,2} \le n_{2,1}$: the parameter $n_{q,i}$ counts the number of $i$s in row~$q$. The action of~$s_1$ is here as follows:
\[s_1(n_{1,2},n_{2,1},n_{2,2}) = (n_{1,2},n_{2,2}+n_{1,2},n_{2,1}-n_{1,2}).\]
\end{example}

We will need some more properties of the letter permuting operators. Recall the {row} and {column maps} $\R,\C$ from Figure~\ref{P:YT}.

\begin{definition}
The \emph{shape} of a Young tableau is the sequence of the lengths of its rows. The set of tableaux on~$A$ of the same shape~$\lambda$ is denoted by~$\YT_A^{\lambda}$.
\end{definition}

\begin{proposition}\label{PR:SnActionRowsCols}
\begin{enumerate}
\item The operators~$s_i$ above restrict from~$A_n^*$ to the sets $\R(\YT_{A_n}^{\lambda})$ and $\C(\YT_{A_n}^{\lambda})$ of row- and column-readings of Young tableaux of the same shape~$\lambda$.
\item The bijection $\R \colon \YT_{A_n} \to \R(\YT_{A_n})$ intertwines the induced and the restricted versions of the~$s_i$. The same is true for~$\C$.
\item The~$s_i$ commute with the operators $\R \T$ and $\C \T$ on~$A_n^*$.
\item For any $\boldsymbol{w_j} \in \R(\YT_{A_n}^{\lambda_j})$ (or $\C(\YT_{A_n}^{\lambda_j})$), the word $s_i(\boldsymbol{w_1} \ldots \boldsymbol{w_k}) \in A_n^*$ decomposes as $\boldsymbol{\widetilde{w_1}} \ldots \boldsymbol{\widetilde{w_k}}$, where each $\boldsymbol{\widetilde{w_j}}$ lies in $\R(\YT_{A_n}^{\lambda_j})$ (respectively, $\C(\YT_{A_n}^{\lambda_j})$). 
\item The row and column versions of the above decomposition coincide. Explicitly, for any $T_j \in \YT_{A_n}$, put $\widetilde{T_j} = \T(\boldsymbol{\widetilde{w_j}})$ and $\widehat{T_j} = \T(\boldsymbol{\widetilde{v_j}})$, where the $\boldsymbol{\widetilde{w_j}}$ and the $\boldsymbol{\widetilde{v_j}}$ are computed as in Point~4 from $\boldsymbol{w_j} = \R(T_j)$ and $\boldsymbol{v_j} = \C(T_j)$ respectively. The tableaux thus obtained are identical: $\widetilde{T_j} = \widehat{T_j}$ for all~$j$.
\end{enumerate}
\end{proposition}

\begin{proof}
We will treat the row case only, the column case being analogous.

Take a $T \in \YT_{A_n}^{\lambda}$. Assume that $s_i$ acts on~$\R(T)$ by replacing some of its $i$s with $i+1$s; the case of the replacement $i+1 \leadsto i$ is similar. Reassemble $s_i(\R(T))$ into a tableau~$\widetilde{T}$ of shape~$\lambda$, filling each row from left to right, starting from the top row. We will show that~$\widetilde{T}$ is a Young tableau. This yields $s_i \R(T) = \R(\widetilde{T})$, hence Point~1.

We first need an observation, which directly follows from our matching procedure:

\begin{observation}\label{L:Neighbors}
Consider a word $\bw \in A_n^*$ with two consecutive letters~$i$. If in~$s_i(\bw)$ the left~$i$ changes to~$i+1$, then so does the right one. Similarly, if in~$s_{i-1}(\bw)$ the right~$i$ changes to~$i-1$, then so does the left one.
\end{observation}

Return now to our~$\widetilde{T}$, seen as a modification of~$T$. There are only two possible reasons for it not to be a Young tableau:
\begin{itemize}
 \item An~$i$ changed to~$i+1$, while its right neighbor~$i$ remained unchanged. This contradicts Observation~\ref{L:Neighbors}.
 \item An~$i$ changed to~$i+1$, while its top neighbor~$i+1$ remained unchanged. Assume that they live in column~$p$, and in rows~$q+1$ and~$q$ respectively. This means that this~$i+1$ was matched with some~$i$, which has to lie in row~$q+1$ and column $p' < p$. Then in~$T$, in each column between~$p'$ and~$p$, row~$q+1$ contained~$i$ and row~$q$ contained~$i+1$. This implies that our~$i$ from the cell $(p,q+1)$, unmatched by assumption, had to match with the~$i+1$ from the cell $(p',q)$. Contradiction.
\end{itemize} 
One concludes that $\widetilde{T}$ is a Young tableau, as announced.

The modified tableau $\widetilde{T}$ will serve us once more. Recall that it satisfies the relation $s_i \R(T) = \R(\widetilde{T})$. By Theorem~\ref{T:Knuth}, one has $\widetilde{T} = \T  \R(\widetilde{T}) = \T s_i \R(T)$. The action of~$s_i$ on Young tableaux is induced by its action on~$A_n^*$, implying $\T s_i = s_i \T$, and hence $\widetilde{T} = s_i \T \R(T) = s_i(T)$. So $s_i \R(T) = \R(\widetilde{T}) = \R s_i(T)$. This proves Point~2. Point~3 follows since~$s_i$ commutes with~$\T$.

The proof of Point~4 repeats verbatim that of Point~1. It remains to check Point~5. According to Theorem~\ref{T:Knuth}, for all~$j$ the words $\boldsymbol{w_j} = \R(T_j)$ and $\boldsymbol{v_j} = \C(T_j)$ are Knuth-equivalent. Lemma~\ref{L:SiKnuth} then yields $\boldsymbol{\widetilde{w_j}} \sim \boldsymbol{\widetilde{v_j}}$. Again by Theorem~\ref{T:Knuth}, one obtains the desired equality $\T(\boldsymbol{\widetilde{w_j}}) = \T(\boldsymbol{\widetilde{v_j}})$.
\end{proof}

Now we can extend the $S_n$-action given by the operators~$s_i$ from~$\YT_{A_n}$ to $(\YT_{A_n} \times \NN_{0})^{\times k}$. Concretely, put 
$s_i((T_1,\alpha_1), \ldots, (T_k,\alpha_k)) = ((\widetilde{T_1},\alpha_1), \ldots, (\widetilde{T_k},\alpha_k))$, 
where the $\widetilde{T_j}$ are computed from the $T_j$ by the last point of Proposition~\ref{PR:SnActionRowsCols}. That proposition, together with Theorem~\ref{T:SnAction}, imply that these~$s_i$ define an $S_n$-action on $(\YT_{A_n} \times \NN_{0})^{\times k}$.

\begin{theorem}\label{T:ActionCompat}
Take integers $n,k > 0$, and consider the power $(\YT_{A_n} \times \NN_{0})^{\times k}$ of the set of $\NN_{0}$-decorated Young tableaux on~$A_n$. It carries two actions: the letter permuting $S_n$-action via the operators~$s_i$ defined above, and the action of the positive braid monoid~$B_k^+$ via the braiding $\osigma_R$ from Theorem~\ref{T:BraidingPlactic}. The two actions commute. Similarly, the operators~$s_i$ and $\osigma_C$ yield pairwise commuting $S_n$- and $B_k^+$-actions.
\end{theorem}

As a consequence, one obtains two actions of the direct product $S_n \times B_k^+$ on $(\YT_{A_n} \times \NN_{0})^{\times k}$. Pursuing Remark~\ref{R:RestrictedTableaux}, one can restrict these actions to those of $C_n \times B_k^+$ on $(\YT_A^{r \le m})^{\times k}$ and on $(\YT_A^{c \le m})^{\times k}$ respectively, for any $m>0$. 

\begin{proof}
We treat only the row case here, the column case being similar. We also omit the $\NN_{0}$-decorations, which demand a laborious but straightforward book-keeping.

Take Young tableaux $T_j \in \YT_{A_n}^{\lambda_j}$. By Proposition~\ref{PR:SnActionRowsCols}, one has $s_i(T_1, \ldots, T_k) = (\widetilde{T_1}, \ldots, \widetilde{T_k})$, with $\widetilde{T_j} \in \YT_{A_n}^{\lambda_j}$. Further, as explained in Theorem~\ref{T:BraidingPlactic}, a generator~$b_l$ of~$B_k^+$ sends $(T_1, \ldots, T_k)$ to $(T_1, \ldots, T'_{l+1},T'_{l},\ldots, T_k)$, and the tableau $T'_{l+1}$ can be placed on top of $T'_{l}$ to obtain $T'_{l+1} \ast T'_{l}$. By Proposition~\ref{PR:SnActionRowsCols} and Lemma~\ref{L:SiKnuth}, one has $s_i(T_1, \ldots, T'_{l+1},T'_{l},\ldots, T_k) = (\widetilde{T_1}, \ldots, \widetilde{T'_{l+1}},\widetilde{T'_{l}},\ldots, \widetilde{T_k})$, and $\widetilde{T'_{l+1}}$ placed on top of $\widetilde{T'_{l}}$ yields $\widetilde{T'_{l+1}} \ast \widetilde{T'_{l}}$. The diagram below summarizes all our notations.
\[\xymatrix@!0 @R=1.2cm @C=8cm{
    (T_1, \ldots, T_l,T_{l+1}, \ldots, T_k) \ar@{|->}[r]^{s_i} \ar@{|->}[d]^{b_l}  & (\widetilde{T_1}, \ldots, \widetilde{T_l},\widetilde{T_{l+1}}, \ldots,\widetilde{T_k})\\
    (T_1, \ldots, T'_{l+1},T'_{l},\ldots,T_k) \ar@{|->}[r]^{s_i} & (\widetilde{T_1}, \ldots, \widetilde{T'_{l+1}},\widetilde{T'_{l}},\ldots,\widetilde{T_k})}\]
The theorem follows if we show that the map~$b_l$ on the right completes this diagram into a commutative square. For this, one needs to compare $\widetilde{T'_{l+1}} \ast \widetilde{T'_{l}}$ and $\widetilde{T_{l}} \ast \widetilde{T_{l+1}}$. Recall that the row map~$\R$ intertwines the $\ast$-product and the concatenation product. Lemma~\ref{L:SiKnuth} then yields 
\[\R(\widetilde{T_{l}} \ast \widetilde{T_{l+1}}) = \R(\widetilde{T_{l}})\R(\widetilde{T_{l+1}}) \sim \R(\widetilde{T'_{l+1}}) \R(\widetilde{T'_{l}}) = \R(\widetilde{T'_{l+1}} \ast \widetilde{T'_{l}}).\]
The injectivity of $\R \colon \YT_{A_n} \to \Pl_{A_n}$ implies $\widetilde{T_{l}} \ast \widetilde{T_{l+1}} = \widetilde{T'_{l+1}} \ast \widetilde{T'_{l}}$, as desired.
\end{proof}

\begin{remark}
It could seem more natural to extend the~$s_i$ to $(\YT_{A_n} \times \NN_{0})^{\times k}$ diagonally. However, this extension is no longer compatible with the $B^+_k$-actions. Take for instance $k=2$ and one-cell tableaux $\ytableaushort{1}$ and $\ytableaushort{2}$. One computes
\begin{align*}
\sigma_R (s_1 \times s_1)(\ytableaushort{1},\ytableaushort{2}) &= \sigma_R (\ytableaushort{2},\ytableaushort{1}) = (\ytableaushort{2},\ytableaushort{1}),\\
(s_1 \times s_1) \sigma_R (\ytableaushort{1},\ytableaushort{2}) &=(s_1 \times s_1)(\emptyr,\ytableaushort{12}) =(\emptyr,\ytableaushort{12}).
\end{align*}
\end{remark}

\section{Braided cohomology: generalities}\label{S:BrHom}

We now attack the cohomological part of the paper. This section briefly reviews braided cohomology theory for a pseudo-unital idempotent braided set $(X,\sigma,1)$. The resulting cohomology groups turn out to compute the Hochschild cohomology of the structure monoid $\Mon(X,\sigma,1)$. Moreover, this identification respects cup products. For details and proofs, see~\cite{LebedIdempot}. In Section~\ref{S:BrHomPlactic}, this approach is applied to the column braiding and plactic monoids: the former is used to compute the cohomology of the latter.

Braided cohomology works as follows. Fix a PUIBS $(X,\sigma,1)$ and a commutative unital ring~$\kk$. For any $k \in \NN$, consider the $\kk$-module of \emph{critical\footnote{Our terminology is borrowed from the algebraic discrete Morse theory \cite{Skol,JoWe}, which is behind our cohomology comparison results.} $k$-cochains}
\[CrC^k = \{\, f \colon X^{\times k} \to \kk \,|\, f(\ldots, 1, \ldots) = 0; f(\ldots, x,y, \ldots) = 0 \text{ when } \sigma(x,y) = (x,y) \}.  \]
Complete this with $CrC^0 = \kk$. Choose a \emph{braided character} of $(X,\sigma,1)$, i.e., a map $\varepsilon \colon X \to \kk$ satisfying $\varepsilon(1) = 1_{\kk}$ and $\varepsilon(x)\varepsilon(y) = \varepsilon(y')\varepsilon(x')$ whenever $\sigma(x,y) = (y',x')$. The simplest example is the \emph{constant braided character} $\varepsilon_1(x):=1_{\kk}$. As usual, the positive braid monoid~$B^+_k$~\eqref{E:Bn} acts on $X^{\times k}$ via~$\sigma$. Consider the \emph{braided differentials} 
\begin{align*}
d_{br}^{k} = \sum\nolimits_{i=1}^{k}(-1)^{i-1}(d^{k;i}_{l} - d^{k;i}_{r}) \colon& CrC^{k-1} \to CrC^{k}, \quad k > 0,\\
\text{where }\quad (d^{k;i}_{l}f) (x_1,\ldots,x_k) &= \varepsilon(x'_i)f(x'_1,\ldots,x'_{i-1},x_{i+1},\ldots,x_k),\\
x'_i x'_1 \ldots x'_{i-1} &= b_1 \cdots b_{i-1}(x_1 \ldots x_i),\\
(d^{k;i}_{r}f) (x_1,\ldots,x_k) &= f(x_1,\ldots,x_{i-1},x''_{i+1},\ldots,x''_k)\varepsilon(x''_i),\\
x''_{i+1} \ldots x''_k x''_i &= b_{k-i} \cdots b_1(x_i \ldots x_k).
\end{align*}
The superscript $i$ indicates the component of our $k$-tuple which should be pulled to the left or to the right, according to the subscript being~$l$ or~$r$. The chased component acts on, and is acted on by, everyone it crosses when moving, using the braiding~$\sigma$. After getting to the very left/right, it becomes an argument for~$\varepsilon$, while the remaining $(k-1)$-tuple is fed to~$f$. See~\cite{LebedIdempot} for more intuitive graphical and shuffle versions of these constructions.

Given two critical $k$-cochains $f \in CrC^p$, $g \in CrC^q$, their \emph{cup product} is defined by
\[ f \smile g (x_1,\ldots,x_{p+q})  = \sum_{s \in Sh_{p,q}} (-1)^{|s|} (f \times g)(T_{s^{-1}} (x_1,\ldots,x_{p+q})). \]
Here the summation runs over all \emph{$(p,q)$-shuffles}, i.e., all permutations $s \in S_{p+q}$ preserving the order of the first~$p$ and the last~$q$ elements: $s(1)<s(2)<\ldots<s(p)$, $s(p+1)<s(p+2)<\ldots<s(p+q)$. Further, taking any shortest words $t_{i_1}t_{i_2}\cdots t_{i_{|s|}}$ representing $s^{-1}$ in terms of elementary transpositions, lift it to $b_{i_1}b_{i_2}\cdots b_{i_{|s|}} \in B^+_{p+q}$. This lift depends on the permutation~$s$ only, and is denoted by $T_{s^{-1}}$. It acts on $X^{\times (p+q)}$ via~$\sigma$. Finally, the product $f \times g$ sends $(x'_1,\ldots,x'_{p+q})$ to $f (x'_1,\ldots,x'_{p})g(x'_{p+1},\ldots,x'_{p+q}) \in \kk$.

One can check that the two operations above are $\kk$-(bi)linear and well defined: the results are indeed critical $k$-cochains. Moreover, as the names suggest, they define a cohomology theory:

\begin{theorem}[\cite{LebedIdempot}]\label{T:BrHom}
The data $(CrC^k, d_{br}^k, \smile)$ above define a differential graded associative algebra, graded commutative up to homotopy. That is,
\begin{enumerate}
\item the differential squares to zero: $d_{br}^{k+1} d_{br}^k = 0$;
\item the cup product is associative, with the unit $1_{\kk} \in \kk = CrC^0$;
\item the differential is a derivation w.r.t. the cup product: 
\[d_{br}^{p+q+1}(f \smile g) = d_{br}^{p+1}(f) \smile g + (-1)^p f \smile d_{br}^{q+1}(g), \qquad f \in CrC^p, g \in CrC^q;\]
\item if the constant braided character $\varepsilon_1$ is used, then there exists a second product $\circ \colon CrC^{p} \otimes CrC^{q} \to CrC^{p+q-1}$ measuring the commutativity defect of~$\smile$:
\[g \smile f - (-1)^{pq} f \smile g = (-1)^{q}d_{br}^{p+q}(f \circ g) + (d_{br}^{p+1} f) \circ g - (-1)^{q}f \circ (d_{br}^{q+1} g).\]
\end{enumerate}
\end{theorem}

The cohomology of the above complex is denoted by $H^*(X,\sigma,1; \kk,\varepsilon)$, and called the \emph{braided cohomology} of $(X,\sigma,1)$ with coefficients in $(\kk,\varepsilon)$. 

By the theorem, the cup product induces an associative graded product in cohomology, for which we keep the name \emph{cup product} and the notation~$\smile$. It is commutative in the case of the constant character.

The braided cohomology construction and its properties can be extended to much more general YBE solutions. As coefficients, one can take any \emph{braided bimodule} instead of~$\kk$. Further, this theory admits a dual homological version. See~\cite{Lebed1, LebedIdempot}.

Consider next the reduced structure monoid $M = \oMon(X,\sigma,1)$. Our $\varepsilon$, extended to~$M$ multiplicatively, becomes a monoid character, still denoted by~$\varepsilon$. The \emph{Hoch\-schild cohomology} $H^*(M; \kk,\varepsilon)$ of~$M$ with coefficients in $(\kk,\varepsilon)$ is computed by the complex
\begin{align*}
&HC^k = \{\, f \colon M^{\times k} \to \kk \,|\, f(\ldots, 1, \ldots) = 0\}, \qquad C^0 = \kk,\\
&(d_{H}^{k}f)(x_1,\ldots,x_k) = \varepsilon(x_1) f(x_2,\ldots,x_k) -f(x_1x_2,x_3,\ldots,x_k) + \cdots\\
&\hspace*{3.5cm} +(-1)^{k-1}f(x_1,\ldots,x_{k-2}, x_{k-1}x_k)+(-1)^{k}f(x_1,\ldots,x_{k-1}) \varepsilon(x_k).
\end{align*}
Together with the \emph{cup product}
\[ f \smile g (x_1,\ldots,x_{p+q})  = f(x_1,\ldots,x_{p}) g(x_{p+1},\ldots,x_{p+q}), \]
these data enjoy all the properties from Theorem~\ref{T:BrHom}. As explained in~\cite{LebedIdempot}, this classical result for Hochschild cohomology can be recovered by plugging the PUIBS $(M,\sigma_{Ass},1_M)$ into Theorem~\ref{T:BrHom}, where $\sigma_{Ass}(v,w) = (1_M,vw)$.

Now, consider the maps
\begin{align*}
\TotSh_k \colon \qquad X^{\times k} &\to \kk X^{\times k} \hookrightarrow \kk M^{\times k},\\
x_1 x_2 \cdots x_k & \mapsto \sum_{s \in S_{k}} (-1)^{|s|} T_{s} (x_1\ldots x_{k}).
\end{align*}

\begin{theorem}[\cite{LebedIdempot}]\label{T:BrHomIdempot}
The maps~$\TotSh_k$ induce a differential graded algebra map 
\[\TotSh^* \colon (HC^*, d_{H}^*, \smile) \to (CrC^*, d_{br}^*, \smile)\]
between the Hochschild cochain complex of the reduced structure monoid of a PUIBS, and the critical cochain complex of the PUIBS itself, both with coefficients in $(\kk, \varepsilon)$. This map induces a graded algebra isomorphism in cohomology.
\end{theorem}

The isomorphism $\TotSh^*$, as well as the induced isomorphism in cohomology, are called \emph{quantum symmetrizers}. Notation $\TotSh^*$ is abusively used for any of them.

The quantum symmetrizer can be defined for very general YBE solutions. However, it is unknown if $\TotSh^*$ is a quasi-isomorphism in general. This question was raised independently by Farinati and Garc{\'{\i}}a-Galofre \cite{FarinatiGalofre}, and by Dilian Yang \cite[Question 7.5]{YangGraphYBE}. Besides the idempotent case, the answer is known to be positive for involutive braidings and the constant character on a ring~$\kk$ of characteristic zero~\cite{FarinatiGalofre}.

Theorem~\ref{T:BrHomIdempot} and the results of~\cite{FarinatiGalofre} open the way for braided techniques in Hochschild cohomology computation. This was applied in~\cite{LebedIdempot} to factorizable monoids, resulting in a generalized K\"{u}nneth formula. Applications to plactic monoids are discussed in the next section.

\section{Braided cohomology for plactic monoids}\label{S:BrHomPlactic}

Let us return to the plactic monoid. Theorem~\ref{T:PlIsStrMonoid} interprets~$\Pl_A$ as the reduced structure monoid for the column braiding~$\sigma_C$ or the row braiding~$\sigma_R$ for~$A$. Theorem~\ref{T:BrHomIdempot} then identifies the cohomologies of the three structures:

\begin{corollary}\label{C:PlacticAsBraidedCohom}
The braided cohomology $H^*(\Cols_A,\sigma_C,\emptyc; \kk,\varepsilon)$ of the braided set of $A$-columns and the Hochschild cohomology $H^*(\Pl_A; \kk,\varepsilon)$ of~$\Pl_A$ with the same coefficients are isomorphic graded algebras. The same holds for the braided set of $A$-rows.
\end{corollary}

Two types of braided characters are considered here:
\begin{enumerate}
 \item the constant character $\varepsilon_1$, often omitted from notations;
 \item the character $\varepsilon_0$ vanishing on all non-empty columns/rows/words.
\end{enumerate} 
While the first one is common in the set-theoretic YBE framework, the second one is fundamental in the associative world. More generally, a character~$\varepsilon$ for any of the above structures is determined by its values on $a \in A$, which can be chosen arbitrarily. 

Start with~$\varepsilon_1$. Assume the alphabet~$A$ finite. The degree~$1$ part of $H^*(\Pl_A; \kk)$ is easy to compute. The cocycle condition for $f \colon \Pl_A \to \kk$ reads $f(\bw\bww) = f(\bw) + f(\bww)$. Such additive maps form a free $\kk$-module with a basis~$\zeta_a$ indexed by $a \in A$, where $\zeta_a(\bw)$ counts the occurrences of the letter~$a$ in the word~$\bw$ (which is invariant under Knuth equivalence). Since $d_{H}^1$ vanishes, the~$\zeta_a$ also give a basis of $H^1(\Pl_A; \kk)$. The map $\TotSh_*$ being the tautological injection $\Cols_A \hookrightarrow \Pl_A$, one gets a basis~$\xi_a$ of $H^1(\Cols_A,\sigma_C,\emptyc; \kk)$, $a \in A$, given by
\begin{align}\label{E:LetterCount}
\xi_a(x_1 \ldots x_n) = \# \{i \,|\, x_i = a \} 1_{\kk}.
\end{align}

According to Theorem~\ref{T:BrHom}, different~$\xi_a$ anticommute with respect to the cup product. Moreover, one has $\xi_a \smile \xi_a = 0$. Indeed, the definition of the cup product reads here $\xi_a \smile \xi_a (c_1,c_2) =  \xi_a(c_1) \xi_a(c_2) - \xi_a(c_3) \xi_a(c_4)$, where $\sigma_C(c_1,c_2) = (c_3,c_4)$. If both $c_1$ and $c_2$ contain~$a$, then so do $c_3$ and $c_4$, yielding $\xi_a(c_1) \xi_a(c_2) = 1 = \xi_a(c_3) \xi_a(c_4)$. Otherwise at most one of $c_3$ and $c_4$ contains~$a$, hence $\xi_a(c_1) \xi_a(c_2) = 0 = \xi_a(c_3) \xi_a(c_4)$.

Further, the products $\xi_{a_1} \smile \xi_{a_2} \smile \cdots \smile \xi_{a_k}$ for $a_1 < \cdots < a_k$ are linearly independent in $H^*(\Cols_A,\sigma_C,\emptyc; \kk)$. Indeed, evaluate them on all $k$-tuples of columns of the form $a'_1, a'_2 a'_1, a'_3 a'_2 a'_1, \ldots, a'_k \cdots a'_2 a'_1$ with $a'_1 < \cdots < a'_k$. By Observation~\ref{O:SubCol}, the braiding~$\sigma_C$ acts on such $k$-tuples by permutation. As a consequence,
\begin{itemize}
 \item the evaluation yields~$1_{\kk}$ if $a_i = a'_i$ for all~$i$, and~$0$ otherwise;
 \item both $d^{k;i}_{l}$ and $d^{k;i}_{r}$ simply remove the $i$th component of such a $k$-tuple, without affecting the remaining ones; thus coboundaries $d_{br}^k f$ vanish on such $k$-tuples.
\end{itemize} 

We have just proved
\begin{theorem}\label{T:CohomPlactic}
Let~$A$ be a finite ordered set, and~$\kk$ a commutative unital ring. The exterior algebra~$\Lambda (\kk A)$ injects, as an algebra, into the Hochschild cohomology $H^*(\Pl_A; \kk)$ of the plactic monoid on~$A$, via the map $a \mapsto \zeta_a$.
\end{theorem}

Formula~\ref{E:LetterCount} also defines an $A$-labeled basis~$\theta_a$ of $H^1(\Rows_A,\sigma_R,\emptyr; \kk)$. The above theorem and Corollary~\ref{C:PlacticAsBraidedCohom} guarantee that the~$\theta_a$ freely generate an exterior subalgebra of $H^*(\Rows_A,\sigma_R,\emptyr; \kk)$. This statement is much more difficult to prove by direct combinatorial methods, the row braiding being more delicate than its column cousin.

A natural question is whether the theorem actually describes the whole of $H^*(\Pl_A; \kk)$.

\begin{conjecture}
The injection from Theorem~\ref{T:CohomPlactic} is in fact an algebra isomorphism $\Lambda (\kk A) \simeq H^*(\Pl_A; \kk)$.
\end{conjecture}

This result was stated as a theorem in~\cite{Lopatkin}, but surjectivity was not proved there. Using elementary combinatorics of columns, we checked surjectivity in small degree. An argument working in all degrees is still missing. A pleasant consequence of the conjecture would be cohomology vanishing in degree $> |A|$.

For~$\varepsilon_0$, the cohomology algebra is more intricate. Its description will involve the maps $f_{c_1, \ldots, c_k} \in CrC^k (\Cols_A,\sigma_C,\emptyc)$, where the non-empty columns $c_i$ satisfy $\sigma_C(c_i,c_{i+1}) \neq (c_i,c_{i+1})$ for all~$i$; such $k$-tuples of columns are called \emph{critical}. The value of $f_{c_1, \ldots, c_k}$ is~$1_{\kk}$ on $(c_1, \ldots, c_k)$, and~$0$ on other arguments. Given a cocycle $f \in CrC^k$, we keep notation~$f$ for both its braided cohomology class, and the Hochschild cohomology class corresponding to it via the quantum symmetrizer.  We will also use the $\kk$-module $A^* = \Map(A,\kk)$ and, more generally, $\CCol^* = \Map(\CCol,\kk)$ for any collection of columns~$\CCol$.

\ytableausetup{boxsize=1em}
\begin{theorem}\label{T:CohomPlactic2}
Let~$A$ be an ordered set, and~$\kk$ a commutative unital ring. Consider the Hochschild cohomology $H^* := H^*(\Pl_A; \kk, \varepsilon_0)$ of the plactic monoid on~$A$ with coefficients in $(\kk, \varepsilon_0)$.
\begin{enumerate}
\item $H^*$ is non-zero in every degree, unless $|A| \le 2$.
\item There is a $\kk$-linear isomorphism $H^1 \simeq A^*$.
\item There is a $\kk$-linear isomorphism $H^2 \simeq (\CCol_A^{\wedge})^*$, where $\CCol_A^{\wedge}$ is the collection of $A$-column couples $(\ytableaushort{a},\ytableaushort{b,d})$ with $a \le b$. 
\item The cup product vanishes on $H^1$. 
\ytableausetup{boxsize=.5em}
\item For a two-letter alphabet $A= \{1,2\}$, $H^1$ has a basis $f_{\ytableaushort{{\scriptstyle 1}}},f_{\ytableaushort{{\scriptstyle 2}}}$; $H^2$ has a basis $f_{\ytableaushort{{\scriptstyle 1}},\ytableaushort{{\scriptstyle 2},{\scriptstyle 1}}}, f_{\ytableaushort{{\scriptstyle 2}},\ytableaushort{{\scriptstyle 2},{\scriptstyle 1}}}$; $H^3$ has a basis $f_{\ytableaushort{{\scriptstyle 2}},\ytableaushort{{\scriptstyle 1}},\ytableaushort{{\scriptstyle 2},{\scriptstyle 1}}}$. In higher degrees, the cohomology vanishes. The only non-zero products amongst these generators are 
\[ f_{\ytableaushort{{\scriptstyle 2}}} \smile f_{\ytableaushort{{\scriptstyle 1}},\ytableaushort{{\scriptstyle 2},{\scriptstyle 1}}} = - f_{\ytableaushort{{\scriptstyle 2}},\ytableaushort{{\scriptstyle 2},{\scriptstyle 1}}} \smile f_{\ytableaushort{{\scriptstyle 1}}}  = f_{\ytableaushort{{\scriptstyle 2}},\ytableaushort{{\scriptstyle 1}},\ytableaushort{{\scriptstyle 2},{\scriptstyle 1}}}.\]
\end{enumerate}
\end{theorem}
\ytableausetup{boxsize=1em}

\begin{remark}
In the last point we obtain a cup product in cohomology which is very far from being graded commutative.
\end{remark}

\begin{proof}
By Theorem~\ref{T:BrHomIdempot}, we can compute the braided cohomology $H^*(\Cols_A,\sigma_C,\emptyc; \kk, \varepsilon_0)$ to get the desired Hochschild cohomology. Observe that for our coefficients, the ``left'' terms $d^{k;i}_{l}f$ of the braided differential vanish on critical $k$-tuples $c_1,\ldots,c_k$. Indeed, an empty column appears as one of the components of $\sigma_C (c,c') = (c'',c''')$ if and only if the columns~$c$ and~$c'$ are \emph{gluable}, i.e., merge into the column $c'' = c c'$; in this case $c''' = \emptyc$. So the column $c'_i$ from $d^{k;i}_{l}f(c_1,\ldots,c_k) = \varepsilon(c'_i) f(c'_1,\ldots, c'_{i-1}, c_{i+1},\ldots, c_k)$ cannot be empty, yielding $\varepsilon(c'_i) = 0$. We will discard the terms $d^{k;i}_{l}f$ in what follows.
\begin{enumerate}
\item Assume $k \ge 3$; the remaining cases follow from subsequent points. Let~$A$ contain three distinct letters, say $1 < 2 < 3$. Consider the $k$-tuple of columns 
\[\overline{c} = (\ytableaushort{2},)\ytableaushort{1},\ytableaushort{3,2},\ytableaushort{1},\ytableaushort{3,2}, \ldots,\ytableaushort{1},\ytableaushort{3,2};\]
here the leading column $\ytableaushort{2}$ is added only for odd~$k$. This is clearly a critical $k$-tuple. We will show that the class of $f_{\overline{c}}$ is non-zero in~$H^k$. To check that $f_{\overline{c}}$ is a cocycle, we will prove that the definition of $d^{k+1}_{br}f$ on a critical $(k+1)$-tuple of columns never involves evaluating~$f$ on~$\overline{c}$. Indeed, suppose $f(\overline{c}) = d^{k+1;i}_{r}f(c_1,\ldots,c_{k+1})$. The definition of~$d^{k+1;i}_{r}$ implies $\sigma_C (c_i,c_{i+1}) = (c',c'')$, where $c'$ must be a column of~$\overline{c}$. It contains~$c_i$ by Observation~\ref{O:Comparison}, in a strict way since $\sigma_C (c_i,c_{i+1}) \neq (c_i,c_{i+1})$. Hence $c'$ has at least two entries, and must be $\ytableaushort{3,2}$. The only possibilities for its subcolumn~$c_i$ are $\ytableaushort{2}$ or $\ytableaushort{3}$. Since in~$\overline{c}$ the column preceding $\ytableaushort{3,2}$ is always $\ytableaushort{1}$, one gets $c_{i-1} = \ytableaushort{1}$. But then $\sigma_C (c_{i-1},c_{i}) = (c_{i-1},c_{i})$, which contradicts the criticality. 

To conclude, let us check that all cocycles $d^{k}_{br}g$, $g \in CrC^{k-1}$, vanish on~$\overline{c}$. Indeed, when pulling a column $\ytableaushort{1}$ or $\ytableaushort{2}$ to the right of~$\overline{c}$, we will get an empty column only if we create neighbors $\ytableaushort{2,1},\ytableaushort{3,2}$ while moving. Similarly, while pulling $\ytableaushort{3,2}$ we create neighbors $\ytableaushort{3,2,1},\ytableaushort{3,2}$. In any case, the critical cochain~$g$ vanishes on the resulting non-critical $(k-1)$-tuple.

\item For $f \in CrC^1$, the cocycle condition reads $f(c_1 c_2)=0$ for gluable columns $c_1,c_2$, and is trivial for non-gluable arguments. So a cocycle vanishes on columns with $\ge 2$ cells, and can be arbitrary on one-cell columns. Since $d^{1}_{br} = 0$, we are done.

\item Similarly to Point~1, one shows that a cocycle $f \in CrC^2$ can be defined arbitrarily on couples from~$\CCol_A^{\wedge}$, i.e., on $(\ytableaushort{a},\ytableaushort{b,d})$ with $a \le b$. In particular, critical cochains supported on~$\CCol_A^{\wedge}$ are cocycles. Moreover, a cocycle which is non-zero on~$\CCol_A^{\wedge}$ is non-zero in cohomology, since all coboundaries vanish on~$\CCol_A^{\wedge}$. To conclude, we will present any cocycle $f \in CrC^2$ as $h + d^{2}_{br}g$, where
\begin{itemize}
 \item $h$ coincides with~$f$ on~$\CCol_A^{\wedge}$ and vanishes outside;
 \item $g(a)=0$, and $g(c) = f(a,c')$ for a column~$c$ decomposed as $c=ac'$; here and until the end of the proof notation $a$ stands for a one-cell column, and $c,c',c_i$ for non-empty columns.
\end{itemize} 
By construction, the cocycle $f':=f-(h + d^{2}_{br}g)$ vanishes on~$\CCol_A^{\wedge}$ and on gluable couples $(a,c')$. Consider the cocycle condition $(d^{3}_{br} f')(a,c_2,c_3)$. For a gluable triple, it implies $f'(c,c')=0$ for any gluable arguments. When $a$ and $c_2$ are gluable but $c_2$ and $c_3$ are not, one concludes $f'(c,c')=0$ for non-gluable arguments with $c$ of length $\ge 2$. When $c_2$ and $c_3$ are gluable, $a$ and $c_2$  are not, and $c_2$ is of length $\ge 2$, one gets $f'(a,c_2 c_3) = f'(c'_2, bc_3)$, where $c'_2$ is the column~$c_2$ with some letter~$b$ replaced with~$a$. Columns $c'_2$ and $bc_3$ are not gluable since $a \le b$. Moreover, $c'_2$ and $c_2$ have the same length $\ge 2$. The previous case then yields $f'(a,c_2 c_3) = f'(c'_2, bc_3) = 0$. Summarizing, $f'$ can be non-zero only on non-gluable $(a,c)$ with $c$ of length $\le 2$. But these couples are either from~$\CCol_A^{\wedge}$, or non-critical (if $c$ is one-cell). So $f'$ is the zero map.

\item Since cocycles $f,g \in CrC^1$ vanish on columns with two cells, one has
\[f \smile g (\ytableaushort{a},\ytableaushort{b,d}) =f(\ytableaushort{a})g(\ytableaushort{b,d}) - f(\ytableaushort{p,q})g(\ytableaushort{r}) = 0.\]
Here $a \le b$, and $(\ytableaushort{p,q}, \ytableaushort{r})$ is $(\ytableaushort{b,a}, \ytableaushort{d})$ if $a \le d$, and $(\ytableaushort{a,d}, \ytableaushort{b})$ otherwise.

\item These cohomology computations follow from the previous points, and from the following explicit list of critical $k$-tuples of columns for $A= \{1,2\}$:
\begin{itemize}
\item $k=1$: $\ytableaushort{1}$, $\ytableaushort{2}$, and $\ytableaushort{2,1}$;
\item $k=2$: $(\ytableaushort{2},\ytableaushort{1})$, $(\ytableaushort{1},\ytableaushort{2,1})$, and $(\ytableaushort{2},\ytableaushort{2,1})$;
\item $k=3$: $(\ytableaushort{2},\ytableaushort{1},\ytableaushort{2,1})$;
\item $k \ge 4$: void.
\end{itemize}

Cup products are computed using
\begin{align*}
(f \smile g)(\ytableaushort{2},\ytableaushort{1},\ytableaushort{2,1}) = &f(\ytableaushort{2})g(\ytableaushort{1},\ytableaushort{2,1}) + f(\ytableaushort{2,1})g(\ytableaushort{2},\ytableaushort{1}) \\
&+ f(\ytableaushort{2},\ytableaushort{1})g(\ytableaushort{2,1}) - f(\ytableaushort{2},\ytableaushort{2,1})g(\ytableaushort{1}). \qedhere
\end{align*}
\end{enumerate}
\end{proof}
\ytableausetup{boxsize=1.15em}

\begin{remark}
In the proof we decomposed the $\kk$-module of $2$-cocycles of~$\Pl_A$ as $(\CCol_A^{\wedge})^* \oplus (\Cols^{\ge 2}_A)^*$, where $\Cols^{\ge 2}_A$ is the set of $A$-columns with at least two cells.
\end{remark}

\begin{corollary}
The cohomological dimension of~$\Pl_A$ is
\begin{itemize}
\item $\infty$ if $|A| > 2$;
\item $3$ if $|A| = 2$;
\item $1$ if $|A| = 1$.
\end{itemize}
\end{corollary}

This follows from the theorem and the following observation: for a $2$- or $1$-letter alphabet, there are no critical $k$-tuples of columns with $k \ge 4$ (resp., $k \ge 2$). The $1$-letter case gives the well-known cohomological dimension of the polynomial ring $\kk [x]$.

We hope that the two examples from this section convinced the reader that a detour by the braided cohomology can considerably simplify Hochschild cohomology computations.

\bigskip
\footnotesize
\bibliographystyle{alpha}
\bibliography{biblio}
\end{document}